\documentclass[11pt,leqno]{article}     
\usepackage[utf8]{inputenc}
\usepackage{amssymb,amsmath,latexsym,mathtools}
\usepackage[english]{babel}
\usepackage{dsfont}
\usepackage{hyperref}
\usepackage{parskip}
\usepackage{enumitem}
\numberwithin{equation}{section}  
\usepackage{amsthm}     
\theoremstyle{plain}
\newtheorem{theorem}{Theorem}
\newtheorem*{theorem-non}{Theorem}
\newtheorem{corollary}[theorem]{Corollary}
\newtheorem{lemma}[theorem]{Lemma}
\newtheorem{proposition}[theorem]{Proposition}
\numberwithin{theorem}{section}
\theoremstyle{definition}
\newtheorem{definition}[theorem]{Definition}

\theoremstyle{remark}

\usepackage{tikz-cd}

\title{Derived Hecke action on the trivial cohomology of division algebras}          
\author{Soumyadip Sahu\\ School of Mathematics, Tata Institute of Fundamental Research\\ 1 Homi Bhabha Road, Mumbai, 400005, Maharashtra, India.\\ Email: soumyadip.sahu00@gmail.com}       
\date{} 

\begin{document}
\maketitle
	
\begin{abstract}
This article generalizes Venkatesh's structure theorem for the derived Hecke action on the Hecke trivial cohomology of a division algebra over an imaginary quadratic field to division algebras over all number fields. In particular, we show that the stable submodule of the Hecke trivial cohomology attached to a division algebra is a free module generated by the unit class for the action of the strict derived Hecke algebra. Moreover, the strict derived Hecke algebra possesses a rational form that preserves the canonical rational structure on the stable cohomology during the derived Hecke action. The main ingredients in our improvement are a careful study of the congruence classes in the torsion cohomology of the arithmetic manifold and the author's new result on the reduction map in the $K$-theory of the ring of integers in number fields.       
\end{abstract}
	
	

\section{Introduction}
\label{section1}
In a recent development, Venkatesh and his collaborators introduced degree-raising operations on the cohomology of the arithmetic manifold and studied the effect of these actions in the light of motivic theories \cite{venkatesh}, \cite{harris-venkatesh}, \cite{prasanna-venkatesh}. One of the main constructions in this genre is the $\ell$-adic derived Hecke algebra \cite{venkatesh} that furnishes a graded extension of the Hecke action on the cohomology of the locally symmetric space. Roughly, the derived Hecke algebra acts by transporting Hecke eigenclasses between different cohomological degrees. In the context of tempered cohomological cuspidal representations, Venkatesh constructs a degree-raising action of the Bloch-Kato Selmer group (conjecturally, the motivic cohomology group) on cohomology using the derived Hecke algebra and formulates a conjecture regarding the rationality of this action. He also proves a structure theorem for the derived Hecke action on the trivial Hecke eigensystem associated with a division algebra over an imaginary quadratic field \cite[5]{venkatesh}, illustrating the rationality of the derived Hecke action beyond the tempered setting. The current article aims to study this result for division algebras over an arbitrary number field.    
	
Let $\mathsf{D}$ denote a centrally simple division algebra of dimension $d^2$ over a number field $F$. Suppose that $\mathsf{G} = \text{SL}_1(\mathsf{D})$, i.e., the subgroup of elements in $\mathsf{D}^{\times}$ having reduced norm $1$. Let $Y(\mathsf{K}_f)$ be the arithmetic manifold associated with $\mathsf{G}$ for a torsion-free level structure $\mathsf{K}_f$. We call a class in the singular cohomology of the $Y(\mathsf{K}_f)$ \textit{Hecke trivial} if it is a generalized eigenclass for the trivial Hecke character defined by $T \mapsto \text{deg}(T)$ (Definition~\ref{definition3.3}). Let $H^{*}\big(Y(\mathsf{K}_{f}), \cdot\big)_{\text{tr}}$ be the collection of the Hecke-trivial classes with coefficients in an appropriate ring. There is a general description of the Hecke-trivial cohomology due to Franke \cite{franke}, which prescribes an explicit topological model for the computation of the Hecke-trivial subspace. The point of working with a division algebra is that its arithmetic manifold $Y(\mathsf{K}_{f})$ is compact. Thus, one can invoke Matsushima's formula to see that $H^{*}\big(Y(\mathsf{K}_{f}), \mathbb{C}\big)_{\text{tr}}$ equals the free exterior algebra generated by invariant classes coming from the compact dual (Section~\ref{section3.4}). This description makes the trivial cohomology amenable to study and provides a platform for further $\ell$-adic studies. We restrict the choice of coefficient prime $\ell$ to primes outside an explicit finite subset $S$ during the investigation. In principle, one expects that the theory of $(\mathfrak{g},K)$ cohomology admits a generalization to describe the integral cohomology. Such a formalism should automatically allow us to deal with the torsion classes in the integral cohomology, but we neglect this prospect in the article. Moreover, to avoid technical complications we formulate our results in terms of the \textit{strict} global derived Hecke algebra (Section~\ref{section3.3.2}) which accommodates degree changing operations only at the primes $v$ satisfying $q_v \equiv 1($mod $\ell^n)$ where $q_v$ denotes the size of the residue field $v$. In Venkatesh's theory, the additional congruence condition has roots in the Taylor-Wiles deformation, and the strict avatar is sufficient for most arithmetic applications. Our primary reason for restricting the subset of primes is group cohomological, similar to his formulation of the derived Satake isomorphism (Theorem~3.3 in \cite{venkatesh}). For a rational prime $\ell$, let $\mathbb{T}'_{\ell}$ be the image of $\ell$-adic strict global derived Hecke algebra inside the ring of $\mathbb{Z}_{\ell}$-linear endomorphisms of $H^{*}\big(Y(\mathsf{K}_f), \mathbb{Z}_{\ell}\big)$. We want to generalize the following result: 
	
\begin{theorem} 
\label{theorem1.1}\emph{(Venkatesh)}
Let $\mathsf{D}$ be a centrally simple division algebra over an imaginary quadratic field $F$ and $\mathsf{G} = \emph{SL}_1(\mathsf{D})$. There exists a finite subset of rational primes $S$ so that for each $\ell \notin S$, the following holds: 
\begin{enumerate}[label =(\roman*), align=left, leftmargin=0pt]
\item The trivial cohomology $H^{\ast}\big(Y(\mathsf{K}_f), \mathbb{Z}_{\ell}\big)_{\emph{tr}}$ is stable under the action of $\mathbb{T}'_{\ell}$. Let $\mathbb{T}'_{\ell, \emph{tr}}$ denote the image of $\mathbb{T}'_{\ell}$ in the endomorphism ring of $H^{\ast}\big(Y(\mathsf{K}_f), \mathbb{Z}_{\ell}\big)_{\emph{tr}}$. Then $H^{\ast}\big(Y(\mathsf{K}_f), \mathbb{Z}_{\ell}\big)_{\emph{tr}}$ is a free $\mathbb{T}'_{\ell, \emph{tr}}$-module generated by the multiplicative unit in degree zero. 
\item The $\mathbb{Q}_{\ell}$-algebra of endomorphisms $\mathbb{T}_{\ell, \emph{tr}}' \otimes \mathbb{Q}_{\ell}$ coincides with the $\mathbb{Q}_{\ell}$-algebra generated by $H^{\ast}\big(Y(\mathsf{K}_f), \mathbb{Q}\big)_{\emph{tr}}$ acting on itself by cup product.    
\end{enumerate}
\end{theorem}
	
The key step in the proof of Theorem~\ref{theorem1.1} is to construct suitable derived Hecke operators that act on the cohomology by the cup product with some Hecke trivial class. If $F$ is imaginary quadratic, then for sufficiently large $\ell$ the trivial cohomology $H^{\ast}\big(Y(\mathsf{K}_f), \mathbb{Z}_{\ell}\big)_{\text{tr}}$ is an exterior algebra generated by classes in odd degree where each relevant degree provides exactly one generator. However, there can be many generating classes in a single odd degree for a general number field. Our basic tool to perform this generalization is a careful analysis of the so-called \textit{congruence classes} in the torsion cohomology of the arithmetic manifold. The change of level structure at a place $v$ of $\mathsf{K}_f$ canonically give rise to a homomorphism 
\[H^{\ast}(K_v, \mathds{k}) \to H^{\ast}\big(Y(\mathsf{K}_f), \mathds{k}\big)\]
where $K_v$ is the level at $v$, and $H^{\ast}(K_v, \mathds{k})$ is the continuous cohomology of the profinite group $K_v$ with coefficients in a discretely topologized commutative ring $\mathds{k}$. We refer to the classes in the image of this map as congruence classes at $v$ (Section~\ref{section3.2}). A derived Hecke operator is characterized by a double coset and a suitable congruence class that arises from a covering space attached to the double coset. The congruence classes are Hecke trivial under mild hypotheses and related to the stable cohomology in our setting. There is a standard notion of stability for the general linear group and its classical subgroups that utilizes the natural inclusion $\text{GL}_n \hookrightarrow \text{GL}_{n+1}$ to compare the cohomology of a finite-dimensional group with the cohomology of the limit (Section~\ref{section2}). Nevertheless, one cannot hope to directly apply this procedure for division algebras over a number field. In this article, we utilize the left regular action of the division algebra on itself to construct a variant of the aforementioned stabilization procedure that allows us to define a \textit{stable subspace} in the singular cohomology with appropriate coefficient ring (Section~\ref{section4}). Note that this construction appeals to the algebraic structure associated with the division algebra beyond the anisotropic nature of $\mathsf{G}$. Let $H^{\ast}\big(Y(\mathsf{K}_{f}), \cdot\big)_{\text{st}}$ denote the stable subspace of $Y(\mathsf{K}_{f})$ in the sense described above. It is rather straightforward to pin down the stable classes in the de Rham model by identifying the Hecke trivial summand with the cohomology of the compact dual using Matsushima's formula. In general, for a large prime $\ell$  we show that there is an inclusion $H^{\ast}\big(Y(\mathsf{K}_{f}), \mathbb{Z}_{\ell}\big)_{\text{st}} \subseteq H^{\ast}\big(Y(\mathsf{K}_{f}), \mathbb{Z}_{\ell}\big)_{\text{tr}}$ and also describe the trivial cohomology as a module over the stable subspace. One of the great advantages of our construction of the stable subspace is that it allows us to directly verify that the congruence classes relevant for us indeed lie in the stable submodule, using Quillen's description of the cohomology of the general linear group over a finite field (Section~\ref{section5}). Our key observation in this direction is a comparison between the algebra generated by the congruence classes and the stable subspace, generalizing the main technical lemma of Venkatesh (Lemma~5.3 in \cite{venkatesh}) to division algebras over all number fields. We call a finite place $v$ good for the level $\mathsf{K}_f$ if $\mathsf{G}$ is quasi-split at $v$ and the $v$ component of $\mathsf{K}_f$ is a hyperspecial maximal compact subgroup of $G_v$ (Section~\ref{section1.1}). For a prime $\ell$, let $\text{Cong}_{n}$ is the subalgebra of $H^{\ast}\big(Y(\mathsf{K}_f), \mathbb{Z}/\ell^n\big)$ generated by the congruence classes attached to the places \[\{v \mid \text{$v$ is good and $q_v \equiv 1($mod $\ell^n)$}\}\] where $q_v$ denotes the size of the residue field of $v$.  
	
\begin{proposition}
\label{proposition1.2}
Let $\mathsf{D}$ be a centrally simple division algebra over a number field $F$ and set $\mathsf{G} = \emph{SL}_1(\mathsf{D})$. There exists a finite subset of rational primes $S$ so that for each $\ell \notin S$ and $n \geq 1$ we have $\emph{Cong}_n = H^{\ast}\big(Y(\mathsf{K}_f), \mathbb{Z}/\ell^n\big)_{\emph{st}}$.    
\end{proposition} 
	
The proof of this result appears in Section~\ref{section5} of the article. As mentioned above, it is an easy exercise to check that each class in $\text{Cong}_n$ lies inside the stable submodule. However, to exhibit equality, we need to construct congruence classes that span the whole stable submodule. Our strategy for this step freely borrows from Borel's calculation of the rank of the $K$-group of a ring of integers \cite{gil} and makes use of the properties of the reduction map in $K$-theory:  
\[K_{2j-1}(\mathcal{O}_F) \otimes \mathbb{Z}/\ell \to K_{2j-1}(\mathbb{F}_v) \otimes \mathbb{Z}/\ell; \hspace{.5cm}j \geq 2, \; \mathbb{F}_v =\substack{\text{residue field}\\ \text{at $v$.}}\]
Many authors have already studied the reduction map in the $K$-theory of the ring of integers with a view toward applications in algebraic topology and arithmetic geometry. In a pioneering work on the topic, Arlettaz and Banaszak \cite{crelle1995} showed that given a nontorsion element in an odd $K$-group, there exist infinitely many primes where the image under the reduction map has a large order. One can also use the reduction map to detect linear dependence of the elements and formulate a natural generalization of Erd\"os's support problem for higher $K$-groups; see \cite{comptes2000}, \cite{jnt2005}, \cite{jnt2006}. Recently, the author \cite{reduction} extended the earlier results to describe the reduction pattern of a collection of linearly independent elements, which enables us to construct a spanning set for the stable submodule consisting of congruence classes using elements in $K$-theory. This exercise demonstrates that the stable submodule originates from an integral basis of the $K$-theory, thereby establishing the motivic nature of the stable classes. In general, the trivial cohomology may contain certain unstable $\varepsilon$-classes if $d$ is even that vanish during the stabilization procedure. We, therefore, formulate the structure theorem regarding the derived Hecke action using the stable subspace instead of the full Hecke trivial cohomology. Proofs of all the results described below appear in Section~\ref{section6} of the paper. 
	
\begin{theorem}
\label{theorem1.3}
Let $\mathsf{D}$ be a centrally simple division algebra over a number field $F$ and set $\mathsf{G} = \emph{SL}_1(\mathsf{D})$. There exists a finite subset of rational primes $S$ so that for each $\ell \notin S$, the following holds: 
\begin{enumerate}[label =(\roman*), align=left, leftmargin=0pt]
\item The action of the strict derived Hecke algebra $\mathbb{T}'_{\ell}$ preserves the stable cohomology $H^{\ast}\big(Y(\mathsf{K}_f), \mathbb{Z}_{\ell}\big)_{\emph{st}}$. Let $\mathbb{T}'_{\ell, \emph{st}}$ denote the image of $\mathbb{T}'_{\ell}$ in the endomorphism ring of $H^{\ast}\big(Y(\mathsf{K}_f), \mathbb{Z}_{\ell}\big)_{\emph{st}}$. Then $H^{\ast}\big(Y(\mathsf{K}_f), \mathbb{Z}_{\ell}\big)_{\emph{st}}$ is a free $\mathbb{T}'_{\ell, \emph{st}}$-module generated by the multiplicative unit in degree zero. 
\item The $\mathbb{Q}_{\ell}$-algebra of endomorphisms $\mathbb{T}_{\ell, \emph{st}}' \otimes \mathbb{Q}_{\ell}$ coincides with the $\mathbb{Q}_{\ell}$-algebra generated by $H^{\ast}\big(Y(\mathsf{K}_f), \mathbb{Q}\big)_{\emph{st}}$ acting on itself by cup product.    
\end{enumerate}
\end{theorem}
	
The above formulation of the structure theorem essentially provides an Eisenstein version of the motivic action theory for cusp forms with the motivic cohomology groups replaced by the higher $K$-groups. If $d$ is odd or $F$ is totally imaginary, then the stable submodule equals the full Hecke trivial submodule. Thus, one immediately obtains an assertion about the Hecke trivial summand in these cases: 
	
\begin{corollary}
\label{corollary1.4}
Let $\mathsf{D}$ be a centrally simple division algebra over a number field $F$ so that either $d$ is odd or $F$ is totally imaginary. Then Theorem~\ref{theorem1.1} holds for the group $\mathsf{G} = \emph{SL}_1(\mathsf{D})$.  
\end{corollary}
	
A well-known principle in the arithmetic theory of automorphic forms suggests that each Hecke eigenclass in the cohomology should correspond to some motivic phenomena. We also provide a statement regarding the full Hecke trivial submodule in the general case. The main difference between Theorem~\ref{theorem1.3} and the general case is that in the latter situation, the classes in the minimal degree do not always generate the full eigenspace as a module over the derived Hecke algebra. To describe the derived Hecke module structure of the trivial submodule in a precise manner, we introduce some notation. Set 
\begin{equation*}
\mathcal{M}_{F,d}^{\text{ur}} = \begin{cases}
			\emptyset, & \text{if $d$ is odd;}\\
			\big\{\substack{\text{real places $\sigma$ of $F$}\\ \text{so that $\mathsf{D}$ splits at $\sigma$}}\big\}, & \text{if $d$ is even.}
		\end{cases} 
\end{equation*}   
The description of Hecke trivial cohomology using de Rham model implies that for each $\sigma \in \mathcal{M}_{F, d}^{\text{ur}}$ there exists a class $\varepsilon^{\sigma}_d\in H^{d}\big(Y(\mathsf{K}_{f}), \mathbb{Q}\big)_{\text{tr}}$ arising from the Euler class in the cohomology of the classifying space of $\text{SO}_d$; see Section~\ref{section3.4}. Roughly, these classes capture an extra sign constraint, appearing at each real place that splits an even degree division algebra. 
	
\begin{theorem}
\label{theorem1.5}
Let $\mathsf{D}$ be a centrally simple division algebra over a number field $F$ and set $\mathsf{G} = \emph{SL}_1(\mathsf{D})$. There exists a finite subset of primes $S$ so that for each $\ell \notin S$: 
\begin{enumerate}[label=(\roman*), align=left, leftmargin=0pt]
\item The trivial cohomology $H^{\ast}\big(Y(\mathsf{K}_f), \mathbb{Z}_{\ell}\big)_{\emph{tr}}$ is stable under the action of $\mathbb{T}'_{\ell}$. Let $\mathbb{T}'_{\ell, \emph{tr}}$ be the image of $\mathbb{T}'_{\ell}$ in the endomorphism ring of $H^{\ast}\big(Y(\mathsf{K}_f), \mathbb{Z}_{\ell}\big)_{\emph{tr}}$. Then, $H^{\ast}\big(Y(\mathsf{K}_f), \mathbb{Z}_{\ell}\big)_{\emph{tr}}$ is a free module over $\mathbb{T}'_{\ell, \emph{tr}}$ with basis \[\big\{\prod_{\sigma \in \mathcal{J}}\varepsilon_{d}^{\sigma} \mid \mathcal{J} \subseteq \mathcal{M}_{F,d}^{\emph{un}}\big\}.\] 
\item The $\mathbb{Q}_{\ell}$-algebra of endomorphisms $\mathbb{T}_{\ell, \emph{tr}}' \otimes \mathbb{Q}_{\ell}$ coincides with the $\mathbb{Q}_{\ell}$-algebra generated by $H^{\ast}\big(Y(\mathsf{K}_f), \mathbb{Q}\big)_{\emph{st}}$ acting on $H^{\ast}\big(Y(\mathsf{K}_f), \mathbb{Q}\big)_{\emph{tr}}$ by cup product.    
\end{enumerate} 
\end{theorem}
	
The article is complemented by an appendix that revisits Borel's theorem on the real stable cohomology \cite{borel} of arithmetic groups with a view toward its relevant applications.

\subsection{Notation and conventions}
\label{section1.1}
In this article, $F$ denotes a number field, and $\mathcal{O}_{F}$ is its ring of integers. Let $\mathcal{M}_{F}^{\infty}$, resp. $\mathcal{M}_{F}^{\text{fin}}$, denote the collection of the infinite places, resp. finite places of $F$. Moreover $\mathcal{M}_{F}^{\text{re}}$, resp. $\mathcal{M}_{F}^{\text{im}}$, is subset of all real, resp. complex, embeddings of $F$. If $v$ is a valuation of $F$, then $F_v$ is the completion of $v$ at $F$. Suppose that $v$ is a finite place. Then $\mathcal{O}_{v}$ is the ring of integers of $F_v$, the residue field at $v$ is $\mathbb{F}_v$, and the notation $q_v$ refers to the size of $\mathbb{F}_v$. 
	
Let $\mathsf{G}$ be a connected, reductive algebraic group over a number field $F$. Now suppose $\mathbb{A}_{F}$, resp. $\mathbb{A}_{F, f}$, is the ring of adeles, resp. finite adeles, attached to $F$. For a place $v$ of $F$ we set $G_v := \mathsf{G}(F_v)$. The archimedean component of $\mathsf{G}(\mathbb{A}_F)$ is $\mathsf{G}_{\text{re}} := \mathsf{G}(F \otimes_{\mathbb{Q}} \mathbb{R}) = \prod_{v \in \mathcal{M}_{F}^{\infty}} G_v$. Let  $\mathsf{K}$ be a maximal connected compact subgroup of the real Lie group $\mathsf{G}_{\text{re}}$. Then the (disconnected) symmetric space attached to $\mathsf{G}$ equals $\mathsf{X} := \mathsf{G}_{\text{re}}/\mathsf{K}$. With each choice of an open compact subgroup $\mathsf{K}_f \subseteq \mathsf{G}(\mathbb{A}_{F,f})$ one attaches an arithmetic quotient $Y(\mathsf{K}_f)$ defined as follows: 
\[Y(\mathsf{K}_f) := \mathsf{G}(F) \backslash \mathsf{X} \times \mathsf{G}(\mathbb{A}_{F,f})/\mathsf{K}_f.\]
In general, the arithmetic quotient $Y(\mathsf{K}_f)$ is a disjoint union of locally symmetric spaces and  possesses an orbifold structure. For convenience, we always assume that our compact open subgroup $\mathsf{K}_{f}$ is \textit{neat}. In more detail, if $\{g_{f,j} \mid 1 \leq j \leq m\}$ is a set of representatives for the double cosets $\mathsf{G}(F) \backslash \mathsf{G}(\mathbb{A}_{F,f}) / \mathsf{K}_f$ then each of the arithmetic subgroups $\mathsf{G}(F) \cap g_{f,j} \mathsf{K}_f g_{f,j}^{-1}$ is torsion-free. Moreover one works only with open compact subgroups with a product structure $\mathsf{K}_f = \prod_{v \in \mathcal{M}_{F}^{\text{fin}}} K_{v}$ where $K_v \subseteq G_v$ is an open compact subgroup and $K_v$ is a hyperspecial maximal compact of $G_v$ for all but finitely many $v$. A prime $v$ is \textit{good} for the level structure $\mathsf{K}_f$ if $\mathsf{G}$ is quasi-split at $v$ and $K_v$ is a hyperspecial maximal compact subgroup of $G_v$. This article concerns the study of the Hecke action on the cohomology of the locally symmetric space $Y(\mathsf{K}_{f})$. We compute the homology and cohomology of $Y(\mathsf{K}_{f})$ using the singular homology theory with the coefficients in a commutative ring $\mathds{k}$ unless otherwise stated. One calls a finite place $v$ an \textit{admissible} place for the pair $(\mathsf{K}_{f}, \mathds{k})$ if $v$ is a good place for $\mathsf{K}_{f}$ and the residue characteristic at $v$ is a unit in $\mathds{k}$.

\section{Stable homology and cohomology} 
\label{section2}
This section aims to review the necessary background from the theory of stable homology and cohomology that plays a role in the later calculations. 
	
\subsection{Spaces of primitive and indecomposable elements}
\label{section2.1}
A topological space is \textit{nice} if it is path-connected and its integral homology groups are finitely generated. All the topological spaces appearing in the discussion below are nice.   
	
Let $(X, x_0)$ be a pointed space and $\mathds{k}$ be a field. The diagonal map $\Delta_{X}: X \to X \times X$ turns $H_{*}(X,\mathds{k})$ into a unital $\mathds{k}$-coalgebra. Moreover $H^{*}(X,\mathds{k})$ possesses a natural augmented $\mathds{k}$-algebra structure stemming from the cup product. Let $P_{*}(X,\mathds{k})$, resp. $I^{*}(X,\mathds{k})$, denote the space of primitive elements, resp. indecomposable elements, attached to the homology, resp. cohomology, of $X$ \cite[20.2]{may-ponto}. The natural pairing between homology and cohomology induces a duality pairing between $P_{*}(X,\mathds{k})$ and $I^{*}(X,\mathds{k})$. Recall the Hurewicz homomorphism  $\text{Hur}_{j}: \pi_{j}(X,x_0) \to H_{j}(X, \mathbb{Z})$. One constructs a $\mathds{k}$-linear Hurewicz map using the universal coefficient theorem for homology: 
\[\text{Hur}_{j,\mathds{k}}: \pi_{j}(X,x_0)\otimes_{\mathbb{Z}} \mathds{k} \to H_{j}(X, \mathds{k}). \hspace{.2cm}(\substack{j \geq 1})\] 
	
\begin{lemma}
\label{lemma2.1}
$\emph{Im}(\emph{Hur}_{j,\mathds{k}}) \subseteq P_{j}(X,\mathds{k})$. 
\end{lemma}
	
\begin{proof}
The proof is identical to the integral version and relies on the result that $H^{i}(S^j, \mathds{k}) = \{0\}$ where $S^j$ is the $j$-sphere and $0 < i < j$.  
\end{proof}
	
Now suppose $X$ is an associative $H$-space. Then the canonical multiplication and comultiplication maps attached to cohomology turn $H^{*}(X,\mathds{k})$ into a connected, commutative Hopf algebra of finite type. Moreover, the celebrated Cartan-Serre theorem in rational homotopy theory asserts that \[\text{Hur}_{j, \mathbb{Q}}: \pi_{j}(X,x_0) \otimes \mathbb{Q} \to P_{j}(X,\mathbb{Q}) \hspace{.5cm}(\substack{j \geq 1})\] 
is an isomorphism \cite[9.2]{may-ponto}. We next provide a mod $\ell$ version of this result. 
	
\begin{lemma}
\label{lemma2.2}
Let $(X, x_0)$ be an associative $H$-space and $j \geq 1$. There exists a positive integer $C(X,j)$ divisible by the size of $\pi_{j}(X,x_0)_{\emph{tor}}$ so that for each prime $\ell \nmid C(X,j)$ the mod $\ell$ Hurewicz map 
\[\emph{Hur}_{j, \mathbb{Z}/\ell}: \pi_{j}(X, x_0) \otimes \mathbb{Z}/\ell \to P_{j}(X, \mathbb{Z}/\ell)\]
is an isomorphism. 
\end{lemma}
	
\begin{proof}
The proof of the statement above is an easy exercise using the standard structure theorem on $H^{\ast}(X, \mathbb{Q})$. For details, see Theorem~5.3 in \cite{reduction}. The steps in the demonstration yield an effective construction of $C(X,j)$ provided one knows an explicit set of algebra generators for $H^{\ast}(X, \mathbb{Q})$.    
\end{proof}
	
The main class of topological spaces relevant to this article stems from the inductive limit of the general and special linear groups.  Let $R$ be a commutative ring. We have a natural inclusion of groups  
\[\varphi_{n}^{n+1}: \text{GL}_n(R) \to \text{GL}_{n+1}(R),\; g \mapsto \begin{pmatrix}
		g & 0\\
		0 & 1
	\end{pmatrix},\]
for each $n \geq 1$. Note that $\varphi_{n}^{n+1}\big(\text{SL}_n(R)\big) \subseteq \text{SL}_{n+1}(R)$. Set $\text{GL}_{\infty}(R) = \varinjlim_{n} \text{GL}_{n}(R)$ and $\text{SL}_{\infty}(R) = \varinjlim_{n} \text{SL}_{n}(R)$. Here, $\varphi_{n}^{n+1}$ is compatible with determinants, and there is a well-defined determinant map $\text{GL}_{\infty}(R) \to R^{\times}$. We identify $\text{SL}_{\infty}(R)$ with the kernel of this map so that $\text{SL}_n(R) = \text{GL}_n(R) \cap \text{SL}_{\infty}(R)$. Now suppose $R$ is a noetherian commutative ring with finite Krull dimension. Then van der Kallen's theorem on stability asserts that there exists a positive integer $e$ depending only on $R$ so that for any coefficient ring $\mathds{k}$ and $j \geq 0$ the natural map \[\varphi_{n,j}^{\infty}: H_{j}\big(G_n(R),\mathds{k}\big) \to H_{j}\big(G_{\infty}(R), \mathds{k}\big) \hspace{.3cm}(\substack{G \in \{\text{GL}, \text{SL}\}})\] 
is an isomorphism whenever $n \geq 2j+e$ \cite{kallen}. In practical applications, $R$ is either a finite field or the ring of integers of a number field. Under this additional hypothesis, $H_{j}\big(G_n(R), \mathbb{Z}\big)$ is a finitely generated abelian group for each $j, n \geq 1$. Thus, $BG_{\infty}(R)$ is a nice topological space in our cases of interest. We next recall Quillen's calculation of the cohomology of $\text{GL}_{\infty}$ with values in a finite field and describe a mod $\ell$ Cartan-Serre type theorem in this context. 
	
For any associative ring $R$ the higher $K$-groups of $R$ are defined as 
\[K_j(R) = \pi_{j}\big(B\text{GL}_{\infty}(R)^{+}\big) \hspace{.3cm}(\substack{j \geq 1})\]   
where `$+$' refers to Quillen's plus construction in topology. The space $B\text{GL}_{\infty}(R)^{+}$ possesses a natural $H$-law arising from operations on $\text{GL}_{\infty}$ that turns it into a commutative $H$-group \cite[2]{srinivas}. Let $\mathbb{F}_q$ denote the finite field with $q$ elements and $\ell$ be an odd prime with $\ell \nmid q$. Suppose that $o_{\ell}(q)$ is the smallest positive integer so that $q^{o_{\ell}(q)} \equiv 1($mod $\ell)$. Quillen \cite{quillen} constructs distinguished classes $\{c_{jo_{\ell}(q)} \mid j \geq 1\}$ and $\{e_{jo_{\ell}(q)} \mid j \geq 1\}$ in the cohomology algebra $H^{*}\big(\text{GL}_{\infty}(\mathbb{F}_q), \mathbb{Z}/\ell\big)$ with $\text{deg } c_{jo_{\ell}(q)} = 2jo_{\ell}(q)$ and $\text{deg } e_{jo_{\ell}(q)}=2jo_{\ell}(q)-1$ which give rise to an isomorphism of graded $\mathbb{Z}/\ell$-algebras as follows:  
\begin{equation}
\label{2.1}
H^{*}\big(\text{GL}_{\infty}(\mathbb{F}_q), \mathbb{Z}/\ell\big) \cong \text{Sym}_{\mathbb{Z}/\ell}[c_{jo_{\ell}(q)} \mid j \geq 1] \otimes \Lambda_{\mathbb{Z}/\ell}[e_{jo_{\ell}(q)} \mid j \geq 1].  
\end{equation}
Here $\text{Sym}_{\mathbb{Z}/\ell}[\cdot]$, resp. $\Lambda_{\mathbb{Z}/\ell}[\cdot]$, denotes the free polynomial, resp. exterior, algebra over $\mathbb{Z}/\ell$. The isomorphism \eqref{2.1} implies that 
\begin{equation}
\label{2.2}
I^{2j-1}\big(\text{GL}_{\infty}(\mathbb{F}_q), \mathbb{Z}/\ell\big) = \begin{cases}
			\mathbb{Z} /\ell\, e_{j}, & \text{if $o_{\ell}(q) \mid j$;}\\
			\{0\}, & \text{if $o_{\ell}(q) \nmid j$.}
		\end{cases}\hspace{.3cm}(\substack{j \geq 1})
\end{equation}
One uses the duality between homology and cohomology to discover that 
\[\text{dim}_{\mathbb{Z}/\ell} P_{2j-1}\big(\text{GL}_{\infty}(\mathbb{F}_q), \mathbb{Z}/\ell\big) = \begin{cases}
		1, & \text{$o_{\ell}(q) \mid j$;}\\
		0, & \text{$o_{\ell}(q) \nmid j$.}
	\end{cases}\hspace{.3cm}(\substack{j \geq 1})\]
	
\begin{lemma}
\label{lemma2.3}
Let $2 \leq j < \ell$ and assume that $q^j \equiv 1($\emph{mod} $\ell)$. Then \[\emph{Hur}_{2j-1, \mathbb{Z}/\ell} : K_{2j-1}(\mathbb{F}_q) \otimes \mathbb{Z}/\ell \to P_{2j-1}\big(\emph{GL}_{\infty}(\mathbb{F}_q), \mathbb{Z}/\ell\big)\]
is an isomorphism. 
\end{lemma}  
	
\begin{proof}
The assertion is a consequence of the properties of the first \'etale Chern class map in $K$-theory. For details, see Lemma~5.4 in \cite{reduction}.  
\end{proof}

\subsection{Comparison between $\text{GL}_{\infty}$ and $\text{SL}_{\infty}$}
\label{section2.2}
Let $R$ be a noetherian commutative ring of finite Krull dimension. The goal of the current section is to compare the cohomology rings $H^{*}\big(\text{GL}_{\infty}(R), \mathds{k}\big)$ and $H^{*}\big(\text{SL}_{\infty}(R), \mathds{k}\big)$ for appropriate choices of $\mathds{k}$. Our treatment follows a spectral sequence argument by Borel  \cite{borel2} originally devised for the ring of integers of the number fields with coefficients in a characteristic $0$ field. 
	
For now, let $\mathds{k}$ denote an arbitrary commutative ring. We fix a positive integer $d$ and choose a large enough integer $n_0$ so that for each $n \geq n_0$ the inclusions $\text{GL}_{n}(R) \xhookrightarrow{} \text{GL}_{\infty}(R)$ and $\text{SL}_{n}(R) \xhookrightarrow{} \text{SL}_{\infty}(R)$ induces isomorphisms in homology and cohomology in degrees $\leq d$. Assume that $n \geq n_0$. For $a \in R^{\times}$ let $\Delta(a)$ be the diagonal matrix in $\text{GL}_{n+1}(R)$ whose $(n+1, n+1)$-th entry equals $a$ and all other diagonal entries equal $1$. Set \[Z_{n+1}(R) = \{\Delta(a) \mid a \in R^{\times}\}.\] Here the homomorphism $a \mapsto \Delta(a)$ yields a splitting to $\text{det}_{n+1}: \text{GL}_{n+1}(R) \to R^{\times}$. We consider $\text{SL}_{n}(R)$ as a subgroup of $\text{GL}_{n+1}(R)$ and write 
\[H_{n+1}(R) := \text{SL}_{n}(R) \times Z_{n+1}(R) \cong \text{SL}_{n}(R) \times R^{\times}.\] 
The inclusion $H_{n+1}(R) \xhookrightarrow{} \text{GL}_{n+1}(R)$ induces a morphism of extensions 
\begin{equation*}
\begin{tikzcd}
1 \arrow[r] & \text{SL}_{n}(R)\arrow[r,] \arrow[d, hook, "\varphi^{n+1}_n"] & H_{n+1}(R) \arrow[r] \arrow[d, hook] & R^{\times} \arrow[r] \arrow[d, "\text{Id}_{R^{\times}}"] & 1 \\
1 \arrow[r] & \text{SL}_{n+1}(R)\arrow[r] & \text{GL}_{n+1}(R) \arrow[r, "\text{det}_{n+1}"] & R^{\times} \arrow[r] & 1. 
\end{tikzcd}
\end{equation*} 
Suppose that $\prescript{I}{}{E}^{p,q}_{r}$, resp. $\prescript{II}{}{E}^{p,q}_{r}$, is the cohomological Hochschild-Serre spectral sequence for the trivial module $\mathds{k}$ attached to the bottom, resp. top row of the diagram above. Let $\{f_r\}_{r \geq 2}$ be the morphism between $\prescript{I}{}{E}_{r}^{p,q}$ and $\prescript{II}{}{E}_{r}^{p,q}$ arising from the morphism of extension. By choice of $n$ the restriction map $H^{q}\big(\text{SL}_{n+1}(R), \mathds{k}\big) \to H^{q}\big(\text{SL}_{n}(R), \mathds{k}\big)$ induces isomorphism in degrees $\leq d$. Note that $R^{\times}$ in the top row acts trivially on $H^{\ast}\big(\text{SL}_n(R), \mathds{k}\big)$. Therefore $H^{q}\big(\text{SL}_{n+1}(R), \mathds{k}\big)$ is also a trivial $R^{\times}$-module and $f_{2}^{p,q}$ is an isomorphism for each $0 \leq q \leq d$. It follows that the induced map 
\begin{equation}
\label{2.3}
H^{j}\big(\text{GL}_{n+1}(R), \mathds{k}\big) \to H^{j}\big(H_{n+1}(R), \mathds{k}\big)
\end{equation}
is an isomorphism whenever $0 \leq j \leq d$; cf. \cite[3.3]{user's guide}. One can simplify the cohomology group appearing in the target of \eqref{2.3} using the Kunneth formula. Let $R$ be a commutative ring so that $B\text{SL}_{\infty}(R)$ is a nice topological space and $R^{\times}$ is a finitely generated abelian group. For our applications, $R$ is either the ring of integers of a number field or a finite field. Moreover, assume that $\mathds{k}$ is a field of characteristic $\ell$. We know that \cite[Ch. XII]{cartan-eilenberg}
\begin{equation*}
\begin{aligned}
H^{\ast}(\mathbb{Z}, \mathds{k})  & = \Lambda_{\mathds{k}}[x], \hspace{5cm} \substack{\text{deg}(x) = 1;}\\
H^{\ast}(\mathbb{Z}/N, \mathds{k}) & = \Lambda_{\mathds{k}}[x] \otimes \text{Sym}_{\mathds{k}}[y],\hspace{.3cm}\substack{\text{$\ell \lvert N$ if $\ell > 2$} \\ \text{$\ell^2 \lvert N$ if $\ell = 2$},} \hspace{.5cm}\substack{\text{deg}(x) = 1, \text{deg}(y) = 2}; \\
H^{\ast}(\mathbb{Z}/N, \mathds{k}) & = \text{Sym}_{\mathds{k}}[x],\hspace{2cm}\substack{\text{$\ell = 2$, $2 \lvert N$} \\ \text{and $4 \nmid N$},} \hspace{1.2cm}\substack{\text{deg}(x) = 1;} \\
H^{\ast}(\mathbb{Z}/N, \mathds{k}) & = \mathds{k}, \hspace{3.5cm}\substack{\ell \nmid N.}
\end{aligned}
\end{equation*}
Therefore $H^{\ast}(R^{\times}, \mathds{k})$ is a finitely generated $\mathds{k}$-algebra generated by the elements in degree $\leq 2$.  
	
\begin{theorem}
\label{theorem2.4}
Let $R$ and $\mathds{k}$ be as above. Then the inclusion $\emph{SL}_{\infty}(R) \xhookrightarrow{} \emph{GL}_{\infty}(R)$ induces isomorphisms 
\begin{equation*}
\begin{gathered}
P_d\big(\emph{SL}_{\infty}(R), \mathds{k}\big) \cong P_d\big(\emph{GL}_{\infty}(R), \mathds{k}\big), \\  I^d\big(\emph{GL}_{\infty}(R), \mathds{k}\big) \cong I^d\big(\emph{SL}_{\infty}(R), \mathds{k}\big)
\end{gathered}\hspace{1cm}(\substack{d \geq 3})
\end{equation*}
\end{theorem}
	
\begin{proof}
Let $d$ be an integer $\geq 3$. In the light of duality between the primitive and indecomposable elements, it suffices to verify the statement for $I^{\ast}$. We choose $n \geq n_0$ as in the discussion above. Then the restriction map \eqref{2.3} induces a $\mathds{k}$-algebra isomorphism
\[H^{\ast}\big(\text{GL}_{n+1}(R), \mathds{k}\big)^{\leq d} \xrightarrow{\cong} H^{\ast}\big(H_{n+1}(R), \mathds{k}\big)^{\leq d}\]   
where $(\cdot)^{\leq d}$ denotes the naive truncation at degree $d$. Since $H^{*}(R^{\times}, \mathds{k})$ is generated by the elements in degrees $\leq 2$ the isomorphism above gives rise to the following isomorphism of indecomposable subspaces: 
\[I^{d}\big(\text{GL}_{n+1}(R), \mathds{k}\big) \xrightarrow{\cong}I^{d}\big(\text{SL}_{n}(R), \mathds{k}\big).\]
Proof of the theorem is now clear.     
\end{proof}
	
This comparison allows us to factor the Hurewicz map attached to $\text{GL}_{\infty}(R)$ via $\text{SL}_{\infty}(R)$. In more detail, if the theorem holds then there exists a homomorphism $\overline{\text{Hur}}_{j, \mathds{k}}: K_{j}(R) \otimes_{\mathbb{Z}} \mathds{k} \to P_{j}\big(\text{SL}_{\infty}(R), \mathds{k}\big)$
so that the following diagram commutes:
\begin{equation*}
\begin{tikzcd}[column sep= large]
K_{j}(R) \otimes \mathds{k} \arrow[r, "\overline{\text{Hur}}_{j, \mathds{k}}"] \arrow[dr, "\text{Hur}_{j, \mathds{k}}"] & H_{j}\big(\text{SL}_{\infty}(R), \mathds{k}\big) \arrow[d]\\
& H_{j}\big(\text{GL}_{\infty}(R), \mathds{k}\big). 
\end{tikzcd} \hspace{.5cm}(\substack{j \geq 3})
\end{equation*}
Let the notation be as in the discussion before Lemma~\ref{lemma2.3}. Now suppose $\bar{e}_{\bullet}$, resp. $\bar{c}_{\bullet}$, is the image of $e_{\bullet}$, resp. $c_{\bullet}$ under the natural restriction map  $H^{\ast}\big(\text{GL}_{\infty}(\mathbb{F}_q), \mathbb{Z}/\ell\big) \to H^{\ast}\big(\text{SL}_{\infty}(\mathbb{F}_q), \mathbb{Z}/\ell\big)$. It is clear that $\bar{e}_{\bullet}$-s and $\bar{c}_{\bullet}$-s give rise to nonzero elements in $I^{\ast}\big(\text{SL}_{\infty}(\mathbb{F}_q), \mathbb{Z}/\ell\big)$ in degrees $\geq 3$. In particular an analogue of \eqref{2.2} holds for $\text{SL}_{\infty}$ whenever $j \geq 2$.
	
\begin{corollary}
\label{corollary2.5}
Let the hypotheses be as in Lemma~\ref{lemma2.3}. Suppose that $a \in K_{2j-1}(\mathbb{F}_q) \otimes \mathbb{Z}/\ell$. Then \[\big\langle \overline{\emph{Hur}}_{2j-1, \mathbb{Z}/\ell}(a), \bar{e}_j \big\rangle \neq 0 \iff a \neq 0\] where $\langle \cdot, \cdot \rangle$ denotes the pairing between homology and cohomology. 
\end{corollary}
	
\begin{proof}
Follows from Lemma~\ref{lemma2.3} and the discussion above.  
\end{proof}

\section{Hecke algebra and congruence classes}
\label{section3}
\subsection{Background on Hecke algebra}
\label{section3.1}
Let the notation be as in Section~\ref{section1.1} and $\mathds{k}$ be an arbitrary commutative ring. The $\mathds{k}$-linear \textit{big Hecke algebra} attached to $(\mathsf{G}, \mathsf{K}_{f})$ consists of compactly supported functions $h: \mathsf{G}(\mathbb{A}_{F,f}) \to \mathds{k}$ that are biinvariant under the action of $\mathsf{K}_f$. We endow it with an algebra structure using convolution with respect to the Haar measure on $\mathsf{G}(\mathbb{A}_{F,f})$ normalized so that $\mathsf{K}_f$ has unit volume. Each element of the big Hecke algebra is a finite linear combination of indicator functions of double cosets of $\mathsf{K}_f$. The following discussion briefly recalls the action of the double coset operators on the cohomology of $Y(\mathsf{K}_f)$ with constant coefficients. Let $g_f \in \mathsf{G}(\mathbb{A}_{F,f})$ and consider the double coset operator $[\mathsf{K}_f g_f \mathsf{K}_f]$ where $[\cdot]$ stands for indicator function. Set  $\mathsf{K}_{f,g_f} := \mathsf{K}_{f} \cap g_f\mathsf{K}_{f}g_f^{-1}$. There are two finite covering maps 
\begin{equation*}
\begin{gathered}
\pi_{g_f}: Y(\mathsf{K}_{f, g_f}) \to Y(\mathsf{K}_f), \; [g] \mapsto [g];\\ 
\widetilde{\pi}_{g_f}: Y(\mathsf{K}_{f, g_f}) \to Y(\mathsf{K}_f), \; [g] \mapsto [gg_f] 
\end{gathered}
\end{equation*} 
arising from the natural inclusion $\mathsf{K}_{f, g_f} \hookrightarrow \mathsf{K}_{f}$ and the twisted inclusion $\mathsf{K}_{f, g_f} \hookrightarrow \mathsf{K}_{f}$, $k \mapsto g_f^{-1}kg_f$, respectively. With $[\mathsf{K}_fg_f\mathsf{K}_f]$ one associates an endomorphism  of $H^{\ast}\big(Y(\mathsf{K}_{f}), \mathds{k}\big)$ as follows: 
\[T_{g_f}: H^{\ast}\big(Y(\mathsf{K}_{f}), \mathds{k}\big) \to H^{\ast}\big(Y(\mathsf{K}_{f}), \mathds{k}\big), \hspace{.3cm}T_{g_f} := \pi_{g_f!}\widetilde{\pi}_{g_f}^{\ast}\] 
where $\pi_{g_f!}$ denotes the push-forward along the finite covering map $\pi_{g_f}$ \cite[3.G]{hatcher}. The operator $T_{g_f}$ does not depend on the choice of representative for the double coset. Moreover, the attachment $[\mathsf{K}_fg_f\mathsf{K}_f] \mapsto T_{g_f}$ extends to an action of the big Hecke algebra on the cohomology. We refer to the image of the big Hecke algebra in the endomorphism ring of cohomology as the \textit{Hecke algebra}. The double cosets supported at a finite place $v$, namely $\{T_{g_v}\mid g_v \in G_v\}$, generate the Hecke algebra at $v$. In this article, we typically restrict ourselves to the subalgebra generated by the double cosets supported at admissible places.

For $g_f \in \mathsf{G}(\mathbb{A}_{F,f})$ define $\text{deg}([\mathsf{K}_f g_f \mathsf{K}_f]) := [\mathsf{K}_f : \mathsf{K}_{f,g_f}]$. Note that the covering map $\pi_{g_f}$ has degree $[\mathsf{K}_f : \mathsf{K}_{f,g_f}]$ and 
\[T_{g_f}(\mathds{1}) = [\mathsf{K}_f : \mathsf{K}_{f,g_f}] \mathds{1}\]
where $\mathds{1}$ is the multiplicative identity of the cohomology ring $H^{\ast}\big(Y(\mathsf{K}_{f}), \mathds{k}\big)$. The attachment $[\mathsf{K}_f g_f \mathsf{K}_f] \mapsto \text{deg}([\mathsf{K}_f g_f \mathsf{K}_f])$ gives rise to a character on the big Hecke algebra which descends to a character on the Hecke algebra thanks to the identity displayed above. Thus, we obtain the \textit{trivial character} on the Hecke algebra described by $T \mapsto \text{deg}(T)$. The unit class in degree zero is an eigenclass for this eigensystem. Since our discussion extensively uses the covering space picture, it is desirable to know a practical criterion regarding the invertibility of the index of the coverings in the $\ell$-adic coefficient ring. For a finite place $v$ of $F$ and $g_v \in G_v$ write $K_{v,g_v} = K_v \cap g_v K_v g_v^{-1}$. Note that the index of the covering $\pi_{g_v}$ equals $[\mathsf{K}_{f}: \mathsf{K}_{f, g_v}] = [K_v : K_{v, g_v}]$.

\begin{proposition}
\label{proposition3.1}
Let $v$ be a good place for the level $\mathsf{K}_f$ so that $\mathsf{G}$ splits at $v$. Now suppose $\ell$ is a rational prime satisfying $q_v \equiv 1($\emph{mod} $\ell)$ and $\ell \nmid \lvert W \rvert$ where $W$ is the Weyl group of $\mathsf{G}$. Then the index of $\pi_{g_v}$ is coprime to $\ell$ for each $g_v \in G_v$. 
\end{proposition}
	
Here, the hypothesis that $\mathsf{G}$ is split merely facilitates the technical setting of our proof. However, the restrictions on the coefficient prime $\ell$ are necessary, and without it, the conclusion fails to hold in general. 
	
\begin{proof}
The assumptions on $v$ ensure that $\mathsf{G}$ extends to a split, reductive group scheme $\mathcal{G}$ over $\mathcal{O}_{v}$ such that $K_v = \mathcal{G}(\mathcal{O}_v)$. Moreover, the double coset of $g_v$ admits a representative $a \in \mathsf{A}(F_v)$ where $\mathsf{A}$ is a $F_v$-rational split maximal torus of $\mathsf{G}$ such that $\mathsf{A}$ extends to a maximal split torus of $\mathcal{G}$. Let $\mathcal{M}$ be the centralizer of in $\mathcal{G}$ of the cocharacter in $X_{\ast}(\mathsf{A})$ that corresponds to $a$ under the isomorphism $X_{\ast}(\mathsf{A}) \cong \mathsf{A}(F_v)/ \big(\mathsf{A}(F_v) \cap K_v\big)$ given by $\chi \mapsto \chi(\omega_v)$ where $\omega_v$ is a uniformizer for $\mathcal{O}_v$. We choose a large integer $m$ so that $K_{v, g_v}$ contains the principal subgroup of level $\omega_v^m$. Observe that the image of $K_{v, g_v}$ in $\mathcal{G}(\mathcal{O}_v/\omega_v^m)$ contains $\mathcal{M}(\mathcal{O}_v/\omega_v^m)$. Unfortunately, the index  $[\mathcal{G}(\mathcal{O}_v/\omega_v^m): \mathcal{M}(\mathcal{O}_v/\omega_v^m)]$ is not always a power of $q_v$. Thus, one needs the extra conditions involving the prime $\ell$. Under these additional assumptions Lemma~\ref{lemma3.2} shows that $[\mathcal{G}(\mathbb{F}_v): \mathcal{M}(\mathbb{F}_v)]$ is coprime to $\ell$. As a consequence, the index of $K_{v,g_v}$ in $K_v$ is also coprime to $\ell$.  
\end{proof}   
	
The following lemma utilizes the theory of finite groups of Lie type to compare the size of a reductive group with its maximal torus.
	
\begin{lemma}
\label{lemma3.2}
Let $\mathsf{H}$ be a connected split reductive group over a finite field $\mathbb{F}_q$ and $\mathsf{T}$ be a maximal split torus of $\mathsf{H}$. Suppose that $\ell$ is prime, which does not divide the size of the Weyl group of $\mathsf{H}$ and satisfies $q \equiv 1($\emph{mod} $\ell)$. Then $\frac{\lvert \mathsf{H}(\mathbb{F}_q)\rvert}{\lvert \mathsf{T}(\mathbb{F}_q)\rvert}$ is coprime to $\ell$. 
\end{lemma} 
	
\begin{proof}
We know that $\lvert \mathsf{H} (\mathbb{F}_q) \rvert = \lvert \mathsf{T}(\mathbb{F}_q) \rvert q^N \sum_{w \in W} q^{l(w)}$ where $N$ is the number of positive roots and $l(\cdot)$ is the length on the Weyl group \cite[p.74]{carter85}. Hence $\frac{\lvert \mathsf{H} (\mathbb{F}_q) \rvert}{\lvert \mathsf{T} (\mathbb{F}_q) \rvert} \equiv \lvert W \rvert (\text{mod } \ell)$. It follows that the ratio is coprime to $\ell$. 
\end{proof}
	
Next, we introduce the concept of Hecke-trivial classes, which play a central role in this article.

\begin{definition}
\label{definition3.3}
\begin{enumerate}[label=(\roman*), align=left, leftmargin=0pt]
\item A class $x \in H^{\ast}\big(Y(\mathsf{K}_{f}),\mathds{k}\big)$ is \textit{Hecke-trivial} for a Hecke operator $T$ supported at a finite place $v$ if there exists a positive integer $n = n(x,T)$ so that \[\big(T - \text{deg}(T)\big)^{n}(x) = 0.\]
			
\item We call $x \in H^{\ast}\big(Y(\mathsf{K}_{f}),\mathds{k}\big)$ a \textit{Hecke-trivial} class if it is Hecke-trivial for all the Hecke operators supported at all places $v$ that are admissible for $(\mathsf{K}_{f}, \mathds{k})$. 
\end{enumerate}
\end{definition}

Let $H^{\ast}\big(Y(\mathsf{K}_{f}), \mathds{k}\big)_{\text{tr}}$ denote the collection of all Hecke trivial classes in the cohomology with coefficients in $\mathds{k}$. In general, it is challenging to determine the Hecke trivial submodule; therefore, we introduce an intermediate notion that is more convenient in the context of the derived Hecke algebra.  
	
\begin{definition}
\label{definition3.4}
\begin{enumerate}[label=(\roman*), align=left, leftmargin=0pt]
\item A cohomology class $x \in H^{\ast}\big(Y(\mathsf{K}_{f}), \mathds{k}\big)$ is \textit{super-trivial} at $g_v \in G_v$ if it satisfies $\widetilde{\pi}_{g_v}^{\ast}(x) = \pi_{g_v}^{\ast}(x)$.
			
\item We call $x \in H^{\ast}\big(Y(\mathsf{K}_{f}), \mathds{k}\big)$ a \textit{Hecke super-trivial} class if $x$ is super-trivial for all the elements supported at each admissible finite place for $(\mathsf{K}_f, \mathds{k})$. 
\end{enumerate}
\end{definition} 
	
If $x$ is a super-trivial class for $g_v$ then \[T_{g_v}(x) = \pi_{g_v!} \widetilde{\pi}^{\ast}_{g_v}(x) = \pi_{g_v!} \pi_{g_v}^{\ast}(x) = \text{deg}(T_{g_v})x.\]
In particular, each Hecke super-trivial class is already Hecke trivial. Roughly, the notion of super-triviality provides a nonarchimedean analog of the $\mathsf{G}_{\text{re}}$-invariant classes in the cohomology of the arithmetic manifold.

\subsection{Congruence classes}
\label{section3.2}
This subsection describes the congruence classes in the torsion cohomology of the arithmetic manifold that arise from the change of the level structure $\mathsf{K}_f$. Let $v$ be a fixed finite place of $F$. Suppose that $K_{v,1}$ is an open normal subgroup $K_{v}$ and $\mathsf{K}_1$ is the preimage of $K_{v,1}$ under the projection $\mathsf{K}_f \to K_{v}$. The finite group $K_v/K_{v,1}$ acts on $Y(\mathsf{K}_1)$ from right by $[g] \to [gg_v]$. This action realizes $K_v/K_{v,1}$ as the deck group for the Galois covering $Y(\mathsf{K}_1) \to Y(\mathsf{K}_f)$ and thus gives rise to a unique map in the homotopy category $Y(\mathsf{K}_f) \to B(K_v/K_{v,1})$ where $B(\cdot)$ refers to the classifying space of the finite group $K_v/K_{v,1}$. We pull back along this map to obtain a congruence class homomorphism $H^{\ast}(K_v/K_{v,1}, \mathds{k}) \to H^{\ast}\big(Y(\mathsf{K}_{f}), \mathds{k}\big)$. The congruence class maps attached to different open normal subgroups of $K_v$ are compatible with each other and give rise to a well-defined map 
\[\text{cong}_{v}: H^{\ast}(K_v, \mathds{k}) \to H^{\ast}\big(Y(\mathsf{K}_f), \mathds{k}\big)\]
where the cohomology group appearing in the domain is continuous cohomology with values in the discrete topological ring $\mathds{k}$. We refer to the cohomology classes in the image of $\text{cong}_v$ as the \textit{congruence classes} at $v$. The congruence class homomorphism is natural in $K_v$. In more detail, let $K'_{v}$ be an open subgroup of $K_{v}$ and $\pi: Y(\mathsf{K}'_f) \to Y(\mathsf{K}_f)$ be the covering attached to $K'_{v}$. Then the following diagram commutes: 
\begin{equation}
\label{3.1}
\begin{tikzcd}
H^{\ast}(K_v, \mathds{k}) \arrow[d, "\text{res}"]\arrow[r, "\text{cong}_v"] & H^{\ast}\big(Y(\mathsf{K}_f), \mathds{k}\big)\arrow[d, "\pi^{\ast}"]\\
H^{\ast}(K'_{v}, \mathds{k}) \arrow[r, "\text{cong}_v"] & H^{\ast}\big(Y(\mathsf{K}'_f), \mathds{k}\big). 
\end{tikzcd}
\end{equation}
A result of Venkatesh asserts that under a suitable hypothesis, the congruence classes are automatically Hecke super-trivial.   
	
\begin{lemma}
\label{lemma3.5}
Let $\alpha$ be a congruence class at $v$. Suppose that $w$ is another finite place of $F$. 
\begin{enumerate}[label = (\roman*), align=left, leftmargin=0pt]
\item If $v \neq w$ then $\alpha$ is a Hecke super trivial class for each $g_w \in G_w$. 
\item Now suppose the coefficient ring $\mathds{k}$ is a $\mathbb{Z}_{(\ell)}$-algebra for some $\ell$ so that $\ell \nmid \lvert W \rvert$. Moreover, assume that $v$ is good place for $\mathsf{K}_f$, $\mathsf{G}$ splits at $v$, and $q_v \equiv 1($\emph{mod} $\ell)$. Then $\alpha$ is a super-trivial class for each $g_v \in G_v$. 
\end{enumerate}  
\end{lemma}
	
\begin{proof}
This follows from the proof of Lemma~2.8 in \cite{venkatesh}. Readers may note that the argument given there overlooks the subtlety regarding the index (Proposition~\ref{proposition3.1}), which imposes further restriction for $v = w$. 
\end{proof}

\subsection{Derived Hecke algebra}
\label{section3.3}
\subsubsection{Local derived Hecke algebra}
\label{section3.3.1}
Let $\mathds{k}$ be a \textit{finite} coefficient ring and $v$ be a finite place so that the residue characteristic at $v$ is a unit in $\mathds{k}$. We write $\mathcal{H}(G_v,K_v)_{\mathds{k}}$ for the derived Hecke algebra \cite[2]{venkatesh} attached to $(G_v, K_v)$ with the coefficients in $\mathds{k}$. The derived Hecke algebra is generated by derived Hecke operators of the form $h_{z,\alpha}$ where $z = K_vg_zK_v\in K_v \backslash G_v /K_v$ and $\alpha$ is a cohomology class in $ H^{\ast}\big(K_{v,g_z}, \mathds{k}\big)$. Suppose that $\langle \alpha \rangle$ is the image of $\alpha$ under the congruence class map $H^{\ast}\big(K_{v,g_z}, \mathds{k}\big) \to H^{\ast}\big(Y(\mathsf{K}_{f,g_z}), \mathds{k}\big)$. Then the action of $h_{z, \alpha}$ on $H^{\ast}\big(Y(\mathsf{K}_{f}), \mathds{k}\big)$ equals the following composite:
\begin{equation} 
\label{3.2}
\begin{tikzcd}[row sep= large]
H^{\ast}\big(Y(\mathsf{K}_{f}), \mathds{k}\big) \arrow[r, "\widetilde{\pi}_{g_z}^{\ast}"]\arrow[drr,  dashed, "h_{z, \alpha}"] & H^{\ast}\big(Y(\mathsf{K}_{f, g_z}), \mathds{k}\big) \arrow[r, "-\cup \langle \alpha \rangle"] & H^{\ast}\big(Y(\mathsf{K}_{f,g_z}), \mathds{k}\big)\arrow[d,"\pi_{g_z!}"]\\
& & H^{\ast}\big(Y(\mathsf{K}_{f}), \mathds{k}\big). 
\end{tikzcd}
\end{equation}
In particular, if $\alpha$ is the unit class $\mathds{1}$ then $h_{z,\alpha}$ is the same as $T_{g_z}$. Letting $z$ be the double coset of identity, we see that $h_{z,\alpha}$ acts via $-\cup \langle \alpha \rangle$ on the cohomology of $Y(\mathsf{K}_{f})$. Thus, the endomorphism algebra generated by the derived Hecke operators at $v$ contains the linear operators defined by the cup product with the congruence classes at $v$. 
	
Next, we establish a few basic properties for the derived Hecke operators, which play a role in our work. Our discussion relies on certain functorial properties of the push-forward map attached to a finite covering.  
	
\begin{lemma}
\label{lemma3.6}\emph{(cf. \cite{piacenza})} Let $\pi: E \to X$ be a finite covering map of topological spaces and let $\pi_{!}$ denote the corresponding transfer map in singular cohomology. 
\begin{enumerate}[label = (\roman*), align=left, leftmargin=0pt]
\item Transfer commutes with pullbacks, i.e., if
\begin{equation*}
\begin{tikzcd}
Y \times_{X} E \arrow[r, "\tilde{f}"]\arrow[d, "\tilde{\pi}"] \arrow[dr, phantom, "\ulcorner", very near start] & E \arrow[d, "\pi"]\\
Y \arrow[r, "f"] & X 
\end{tikzcd}
\end{equation*} 
is a pullback of $\pi$ along a continuous map $Y \xrightarrow{f} X$ then $\tilde{\pi}_{!}\circ \tilde{f}^{\ast} = f^{\ast} \circ \pi_{!}$.
\item (Projection formula) For each $x \in H^{\ast}(X, \mathds{k})$ and $y \in H^{\ast}(E, \mathds{k})$ we have 
\[\pi_{!}(\pi^{\ast}(x) \cup y) = x \cup \pi_{!}(y).\]  
\end{enumerate}
\end{lemma}
	
\begin{proof}
At the level of singular chains, the transfer is defined by assigning a singular simplex on $X$, the sum of the finitely many lifts to $E$. For this construction of transfer, the first property is an immediate consequence of the definitions, while the second one follows from the Alexander-Whitney diagonal approximation.   
\end{proof}
	
\begin{corollary}
\label{corollary3.7}
The congruence class map commutes with transfer. In other words, with the notation of \eqref{3.1} the following diagram commutes:   
		
\begin{equation*}
\begin{tikzcd}
H^{\ast}(K_{v}', \mathds{k}) \arrow[r, "\emph{cong}_v"]\arrow[d, "\emph{cores}"] & H^{\ast}\big(Y(\mathsf{K}'_{f}), \mathds{k}\big)\arrow[d, "\pi_{!}"]\\
H^{\ast}(K_{v}, \mathds{k}) \arrow[r, "\emph{cong}_v"] & H^{\ast}\big(Y(\mathsf{K}_{f}), \mathds{k}\big).  
\end{tikzcd}
\end{equation*}
\end{corollary}
	
\begin{proof}
Let $K_{v,1}$ be an open normal subgroup of $K_v$ contained in $K_v'$ and $\mathsf{K}_{1}$ be the level attached to $K_{v,1}$. We consider the commutative diagram 
\begin{equation*}
\begin{tikzcd}
Y(\mathsf{K}_f') \arrow[r, ""]\arrow[d, "\pi"] & B(K_v'/K_{v,1})\arrow[d]\\
				Y(\mathsf{K}_f) \arrow[r, ""] & B(K_v/K_{v,1})
\end{tikzcd}
\end{equation*}
where the horizontal arrows arise from congruence covers and the right vertical arrow arises from the inclusion $K_v'/K_{v,1} \hookrightarrow K_v/K_{v,1}$. Here one realizes $B(K_v'/K_{v,1})$ as $\text{pt} \otimes_{K_v'/K_{v,1}} E(K_v/K_{v,1})$ so that the natural map induced by inclusion stems from the projection $E(K_v/K_{v,1}) \to B(K_v/K_{v,1})$ \cite[p.70]{adem-milgram}. In particular, the right vertical arrow of the diagram above is a covering map. Moreover, the diagram is a pullback diagram for the congruence class map at the bottom of the diagram. For this one needs to note that the pullback of $E(K_v/K_{v,1})$ equals $Y(\mathsf{K}_1)$ and $Y(\mathsf{K}'_f)$ is a quotient of $Y(\mathsf{K}_1)$ by $K_v'/K_{v,1}$. Thus, the assertion is a consequence of the first property of transfer in Lemma~\ref{lemma3.6}.  
\end{proof}
	
We are ready to state our results regarding the derived Hecke operators. 
	
\begin{proposition}
\label{proposition3.8}
Let the notation be as at the beginning of this subsection. 
\begin{enumerate}[label=(\roman*), align=left, leftmargin=0pt]
\item Suppose that $\mathds{1}$ is the unit class in the cohomology of $Y(\mathsf{K}_f)$ and $h_{z,\alpha}$ is a derived Hecke operator at $v$. Then $h_{z,\alpha}(\mathds{1})$ is a congruence class at $v$. 
\item Now suppose $x_0$ is a super-trivial class for $g_z \in G_v$. Then \[h_{z,\alpha} (x_0 \cup x) =  x_0 \cup h_{z,\alpha}(x)\] 
for each $x \in H^{\ast}\big(Y(\mathsf{K}_{f}),\mathds{k}\big)$. In particular, $h_{z,\alpha} (x_0) = x_0 \cup h_{z,\alpha}(\mathds{1})$. 
\end{enumerate}
\end{proposition}
	
\begin{proof}
\begin{enumerate}[label=(\roman*), align=left, leftmargin=0pt]
\item A straight-forward computation using \eqref{3.2} shows that $h_{z, \alpha}(\mathds{1}) = \pi_{g_z!}(\langle \alpha \rangle)$. Now, the assertion is a consequence of Corollary~\ref{corollary3.7}.  
\item We have 
\begin{equation*}
\begin{aligned}
h_{z,\alpha}(x_0 \cup x) & = \pi_{g_z!}\big(\widetilde{\pi}_{g_z}^{\ast}(x_0) \cup \widetilde{\pi}_{g_z}^{\ast}(x) \cup \langle \alpha \rangle\big)\\
& = \pi_{g_z!}\big(\pi_{g_z}^{\ast}(x_0) \cup \widetilde{\pi}_{g_z}^{\ast}(x) \cup \langle \alpha \rangle\big)\\
& = x_0 \cup h_{z, \alpha}(x)
\end{aligned}
\end{equation*}
where the last step follows from the projection formula in Lemma~\ref{lemma3.6}. 
\end{enumerate}       
\end{proof}

\subsubsection{Global derived Hecke algebra}
\label{section3.3.2}
We end the discussion on derived Hecke operators by recalling the basic formalism concerning the global derived Hecke algebra \cite[2.13]{venkatesh}. Let $\ell$ be a rational prime and $n$ be a positive integer. Suppose that \[\mathbb{T}_{\ell, n} \subseteq \text{End}\big(H^{\ast}(Y\big(\mathsf{K}_{f}), \mathbb{Z}/\ell^n\big)\big)\] is the algebra of endomorphisms generated by all the derived Hecke algebras $\mathcal{H}(G_v, K_v)_{\mathbb{Z}/\ell^n}$ where $v$ ranges over admissible primes for $(\mathsf{K}_{f}, \mathbb{Z}/\ell^n)$. Note that the integral homology $H_{\ast}\big(Y(\mathsf{K}_{f}), \mathbb{Z}\big)$ is a finitely generated abelian group and $H^{\ast}\big(Y(\mathsf{K}_{f}), \mathbb{Z}_{\ell}\big) = \varprojlim_{n} H^{\ast}\big(Y(\mathsf{K}_{f}), \mathbb{Z}/\ell^n\big)$. We define the $\ell$-adic \textit{global derived Hecke algebra} to be the subcollection \[\mathbb{T}_{\ell} \subseteq \text{End}\big(H^{\ast}\big(Y(\mathsf{K}_{f}), \mathbb{Z}_{\ell}\big)\big)\] consisting of the elements of the form $\varprojlim t_n$ for some compatible system $(t_n)_{n \geq 1}$ with $t_n \in \mathbb{T}_{\ell, n}$. For practical purposes, one often needs to restrict the primes and the powers of $\ell$ in the above construction. 
	
Let $V: \mathcal{M}_F^{\text{fin}} \to \{0, 1, 2, \ldots\} \cup \{\infty\}$
be a set theoretic function that sends non-admissible places for $(\mathsf{K}_f, \mathbb{Z}/\ell)$ to zero. Set $\mathbb{T}_{\ell,n}^{(V)}$ to the subalgebra of $\mathbb{T}_{\ell,n}$ generated by $\mathcal{H}(G_v,K_v)_{\mathbb{Z}/\ell^n}$ where $v$ satisfies $V(v) \geq n$. Then we obtain a \textit{restricted} global derived Hecke algebra $\mathbb{T}_{\ell}^{(V)}$ by considering compatible systems of the form $(t_n)_{n \geq 1}$ with $t_n \in \mathbb{T}_{\ell,n}^{(V)}$. The most significant example of $V$ for our purpose is 
\[V_0(v) = \begin{cases}
		0, & \text{$v$ not admissible}\\
		\text{largest power of $\ell$ dividing $q_v - 1$}, & \text{$v$ admissible}. 
	\end{cases}\]  
The extra congruence condition $q_v \equiv 1($mod $\ell^n)$ originates from the Taylor-Wiles theory and simplifies lots of technical points (e.g., Proposition~\ref{proposition3.1}) in our 
calculations, as explained in the introduction. One can enlarge the restricted derived Hecke algebra for $V_0$ by adding all classical Hecke operators to define the \textit{strict} global derived Hecke algebra:     
\[\mathbb{T}'_{\ell} := \substack{\text{algebra generated by $\mathbb{T}_{\ell}^{(V_0)}$ and all underived}\\ \text{Hecke operators at admissible places.}}\]
Let $v$ be an admissible place for the pair $(\mathsf{K}_f, \mathbb{Z}_{\ell})$ so that $\mathsf{G}$ splits at $v$. Assume that $\ell$ does not divide the size of the Weyl group of $\mathsf{G}$. Then the Satake isomorphism shows that the usual Hecke algebra attached to $(G_v, K_v)$ with coefficients in $\mathbb{Z}/\ell^n$ is commutative. Moreover, the derived Hecke algebra $\mathcal{H}(G_v,K_v)_{\mathbb{Z}/\ell^n}$ is graded commutative if $q_v \equiv 1(\text{mod }\ell^n)$. In the example of division algebra, the additional split at $v$ condition is automatic, i.e., $\mathsf{G} = \text{SL}_1(\mathsf{D})$ is quasi-split at a finite place $v$ if and only if it splits at $v$. Therefore, the strict global derived Hecke algebra $\mathbb{T}'_{\ell}$ is graded commutative whenever $\ell$ is coprime to the size of the Weyl group of $\mathsf{G}$. 
	
\subsection{Hecke trivial cohomology of division algebras} 
\label{section3.4}
Let $\mathsf{D}$ be a centrally simple division algebra having degree $d^2$ over a number field $F$ and set $\mathsf{G} = \text{SL}_1(\mathsf{D})$ as in the introduction. Note that $\mathsf{G}$ is an anisotropic form over $F$ of the split algebraic group $\text{SL}_d$. In particular 
\[\mathsf{G}(F_{\sigma}) = \begin{cases}
		\text{SL}_{d}(\mathbb{C}), & \text{if $\sigma \in \mathcal{M}_{F}^{\text{im}}$;}\\
		\text{SL}_d(\mathbb{R}), & \text{if $\sigma \in \mathcal{M}_{F}^{\text{re}}$ and $\mathsf{D}$ splits at $\sigma$;}\\
		\text{SL}_{\frac{d}{2}}(\mathbb{H}), & \text{if $\sigma \in \mathcal{M}_{F}^{\text{re}}$ and $\mathsf{D}$ is ramified at $\sigma$.}
	\end{cases}\] 
Let $\mathcal{M}^{\text{ur}}_{F}$, resp. $\mathcal{M}^{\text{ram}}_{F}$, denote the collection of all real places where $\mathsf{D}$ splits, resp. ramifies. Note that $\mathcal{M}^{\text{ram}}_{F}$ is empty if $d$ is odd. For each $\sigma \in \mathcal{M}_{F}^{\infty}$ we define a maximal compact subgroup $K_{\sigma}$ of $\mathsf{G}(F_{\sigma})$ using the standard Hermitian inner products on $F_{\sigma}^{d}$ and $\mathbb{H}_{\sigma}^{\frac{d}{2}}$, i.e.,  
\begin{equation*}
K_{\sigma} = \begin{cases}
			\text{SO}_{d}, & \text{$\sigma \in \mathcal{M}_F^{\text{ur}}$;}\\
			\text{Sp}_{\frac{d}{2}}, & \text{$\sigma \in \mathcal{M}_F^{\text{ram}}$;}\\
			\text{SU}_{d}, & \text{$\sigma \in \mathcal{M}_F^{\text{im}}$.}
		\end{cases}
\end{equation*}
Then $\mathsf{K} := \prod_{\sigma \in \mathcal{M}_{F}^{\infty}} K_{\sigma}$ is a maximal connected compact subgroup of $\mathsf{G}_{\text{re}} = \prod_{\sigma \in \mathcal{M}_{F}^{\infty}} \mathsf{G}(F_{\sigma})$. Thus the symmetric space $\mathsf{X} = \mathsf{G}_{\text{re}}/\mathsf{K}$ is a smooth real manifold of dimension 
\[d(F) := \big(\frac{d(d-1)}{2} - 1\big)\lvert \mathcal{M}_{F}^{\text{re}} \rvert + (d^2 -1)\lvert \mathcal{M}_{F}^{\text{im}} \rvert + d\lvert \mathcal{M}_{F}^{\text{ur}} \rvert.\]
Note that our algebraic group $\mathsf{G}$ is quasi-split at a place $v$ if and only if the division algebra $\mathsf{D}$ splits at $v$. In particular, $\mathsf{G}$ splits at each good prime for $\mathsf{K}_{f}$. Now suppose $\mathcal{O}_\mathsf{D}$ is a fixed maximal order contained in $\mathsf{D}$. For convenience, we additionally assume that $K_v$ is contained in $\mathcal{O}_{\mathsf{D},v}^{\times}$ for each $v \in \mathcal{M}_{F}^{\text{fin}}$. If $v$ is a good place, then there is an isomorphism $\mathsf{D} \otimes_{F} F_v \cong M_d(F_v)$ that maps $\mathcal{O}_{\mathsf{D},v}^{\times}$, resp. $K_v$, onto $\text{GL}_d(\mathcal{O}_{F_v})$, resp. $\text{SL}_d(\mathcal{O}_{F_v})$. Since $\mathsf{G}$ is anisotropic, the arithmetic manifold $Y(\mathsf{K}_f)$ is compact. Moreover, the strong approximation theorem \cite[7.4]{platonov} applies to $\mathsf{G}$ and we have $\mathsf{G}(\mathbb{A}_{F,f}) = \mathsf{G}(F)\mathsf{K}_{f}$. Therefore 
\[\text{$Y(\mathsf{K}_{f}) = \Gamma_{\mathsf{K}_{f}} \backslash \mathsf{X}$ where $\Gamma_{\mathsf{K}_{f}} = \mathsf{G}(F) \cap \mathsf{K}_{f}$.}\] In particular $Y(\mathsf{K}_f)$ is connected. Note that if $d = 1$ then $Y(\mathsf{K}_f)$ is a point and the results announced in Section~\ref{section1} are tautologically true. During the arguments in the latter sections, we always assume $d \geq 2$ to avoid degenerate cases. The rest of this subsection aims to describe the Hecke-trivial submodule attached to the cohomology of $Y(\mathsf{K}_f)$. 
	
We first briefly describe the action of adelic Hecke operators on the classical model. Let $T_{g_v} = [K_{v}g_{v}K_{v}]$ be a double coset operator supported at a good place $v$. Set $\mathsf{K}_{f, g_v} = \mathsf{K}_{f} \cap g_v \mathsf{K}_{f} g_v^{-1}$ and $\Gamma_{g_v} = \mathsf{K}_{f, g_v} \cap \mathsf{G}(F)$.
Since $\mathsf{G}(F)$ is dense in $\mathsf{G}(\mathbb{A}_{F,f})$ there exists $\gamma_{g_v} \in \mathsf{G}(F)$ so that $\gamma_{g_v} \in \mathsf{K}_{f}g_v \cap g_v\mathsf{K}_{f}$. Note that $\Gamma_{g_v} = \Gamma_{\mathsf{K}_{f}} \cap \gamma_{g_v} \Gamma_{\mathsf{K}_f} \gamma_{g_v}^{-1}$. Moreover the covering map $Y(\Gamma_{g_v}) \to Y(\mathsf{K}_f)$ induced by the twisted inclusion $[x] \to [\gamma_{g_v}^{-1}x]$ coincides with the adelic map arising from the twisted inclusion of $\mathsf{K}_{g_v}$ into $\mathsf{K}_f$. Therefore it is possible the encode the action of $T_{g_v}$ using the $\mathsf{G}(F)$-Hecke operators on $Y(\mathsf{K}_{f})$. The Hodge theory of compact Riemannian manifolds implies that the cohomology classes in $H^{\ast}\big(Y(\mathsf{K}_{f}), \mathbb{C}\big)$ admit unique harmonic representatives. Note that the space of harmonic forms on $Y(\mathsf{K}_{f})$ posseses a $\mathsf{G}(F)$-invariant inner product arising from the canonical $\mathsf{G}_{\text{re}}$-invariant metric on $\mathsf{X}$. The adjoint of $T_{g_v}$ with respect the invariant inner product equals $T_{g_v^{-1}}$. But by the Satake isomorphism, the Hecke algebra is commutative at a good place. It follows that the Hecke operators at a good place define normal endomorphisms of $H^{\ast}\big(Y(\mathsf{K}_{f}), \mathbb{C}\big)$. In particular, such Hecke operators are semisimple.  As a consequence, all Hecke operators $T$ supported at a good place $v$ acts on $H^{\ast}\big(Y(\mathsf{K}_f), \mathbb{C}\big)_{\text{tr}}$ as multiplication by $\text{deg}(T)$. 
	
The theory of $(\mathfrak{g}, K)$-modules yields a description of $H^{\ast}\big(Y(\mathsf{K}_{f}), \mathbb{C}\big)$ as the cohomology of the complex $C^{\ast}(\mathfrak{g}, \mathsf{K}, \mathcal{C}^{\infty}(\Gamma_{\mathsf{K}_{f}}\backslash \mathsf{G}_{\text{re}})\big)$ where $\mathfrak{g}$ is the real Lie algebra of the Lie group $\mathsf{G}_{\text{re}}$ and $\mathcal{C}^{\infty}(\cdot)$ is the collection of $\mathbb{C}$-valued smooth functions \cite[Ch.VII]{borel-wallach}. Observe that $\Gamma_{\mathsf{K}_{f}}$ is a cocompact subgroup of $\mathsf{G}_{\text{re}}$. Thus, the inclusion of the constant functions $\mathbb{C} \hookrightarrow \mathcal{C}^{\infty}(\Gamma_{\mathsf{K}_{f}} \backslash \mathsf{G}_{\text{re}})$ yields an injective map 
\begin{equation}
\label{3.3}
H^{\ast}(\mathfrak{g}, \mathsf{K}, \mathbb{C}) \hookrightarrow H^{\ast}(\mathfrak{g}, \mathsf{K}, \mathcal{C}^{\infty}(\Gamma_{\mathsf{K}_{f}} \backslash \mathsf{G}_{\text{re}})) = H^{\ast}\big(Y(\mathsf{K}_{f}), \mathbb{C}\big)
\end{equation}
and its image is the contribution of the trivial representation to the cohomology of $Y(\mathsf{K}_f)$. Let $\Omega_{\mathsf{X}}^{\mathsf{G}_{\text{re}}}$ be the space of $\mathsf{G}_{\text{re}}$-invariant differential forms on $\Omega_{\mathsf{X}}$. There is a canonical identification $\Omega_{\mathsf{X}}^{\mathsf{G}_{\text{re}}} \cong H^{\ast}(\mathfrak{g}, \mathsf{K}, \mathbb{C})$ obtained by evaluating a differential form at the coset of identity. Hence, each element in the image of \eqref{3.3} admits a $\mathsf{G}_{\text{re}}$-invariant differential form representative. As a consequence, each element in the image is a Hecke super-trivial class in the sense of Definition~\ref{definition3.4}. In particular, the image of \eqref{3.3} lies inside $H^{\ast}\big(Y(\mathsf{K}_{f}), \mathbb{C}\big)_{\text{tr}}$. Now, a standard argument using the strong approximation theorem verifies that the image of \eqref{3.3} equals the trivial subspace $H^{\ast}\big(Y(\mathsf{K}_{f}), \mathbb{C}\big)_{\text{tr}}$ \cite[1.5]{venkataramana}.
	
Let $\text{CT}(\mathsf{G}_{\text{re}})$ denote the \textit{compact twin} or \textit{compact dual} attached to the real lie group $\mathsf{G}_\text{re}$. Then \eqref{3.3} above coincides with the Borel map  
\begin{equation}
\label{3.4}
\text{Bo}_{\text{ct}, \Gamma_{\mathsf{K}_{f}}}^{\ast}: H^{\ast}\big(\text{CT}(\mathsf{G}_{\text{re}}), \mathbb{C}\big) \cong H^{\ast}(\mathfrak{g}, K, \mathbb{C}) \to H^{\ast}\big(Y(\mathsf{K}_{f}), \mathbb{C}\big)
\end{equation}
attached to the algebraic group $\text{Res}^{F}_{\mathbb{Q}}\mathsf{G}$ (Appendix~\ref{appendixA.1}). Recall that the Borel map is an algebra homomorphism. Therefore, its image, namely $H^{\ast}\big(Y(\mathsf{K}_f), \mathbb{C}\big)_{\text{tr}}$, is a subalgebra of the full cohomology algebra. The compact twin of the Lie group $\mathsf{G}_{\text{re}}$ is 
\begin{equation*}
\text{CT}\big(\mathsf{G}_{\text{re}}\big) = \prod_{\sigma \in \mathcal{M}_{F}^{\text{ur}}} \frac{\text{SU}_d}{\text{SO}_d} \times \prod_{\sigma \in \mathcal{M}_{F}^{\text{ram}}} \frac{\text{SU}_d}{\text{Sp}_{\frac{d}{2}}} \times \prod_{\sigma \in \mathcal{M}_{F}^{\text{im}}} \text{SU}_d.
\end{equation*}
The cohomology rings of the individual factors are already well-known and described in Appendix~\ref{appendixA.2}. To describe the cohomology of the product, we introduce a convenient notation. For distinct homogeneous cohomology classes $\{x_1, \cdots x_m\}$ let $\Delta_{\mathds{k}}[x_1, \ldots, x_m]$ denote a subalgebra of the cohomology ring which possesses a basis of the form $\{x_{j_1}\cdots x_{j_r} \mid \substack{j_1 < \cdots < j_r \\ 0 \leq r \leq m}\}$. With this notation, the cohomology ring is isomorphic to 
\[\Lambda_{\mathbb{C}}[\alpha_{2p-1}^{\sigma} \mid \substack{2 \leq p \leq d, \sigma \in \mathcal{M}_{F}^\text{im}}] \otimes \Lambda_{\mathbb{C}}[\alpha_{2p-1}^{\sigma} \mid \substack{\text{$2 \leq p \leq d$, $p$ is odd,} \\ \text{and $\sigma \in \mathcal{M}^{\text{re}}_F$}}] \otimes \Delta_{\mathbb{C}}[\varepsilon_{d}^{\sigma} \mid \substack{\text{$\sigma \in \mathcal{M}_{F,d}^{\text{ur}}$}}]\]
where $\alpha$-s and $\varepsilon$-s are are certain canonical classes, and 
\begin{equation*}
\mathcal{M}_{F,d}^{\text{ur}} = \begin{cases}
			\emptyset, & \text{if $d$ is odd;}\\
			\big\{\substack{\text{real places $\sigma$ of $F$}\\ \text{so that $\mathsf{D}$ splits at $\sigma$}}\big\}, & \text{if $d$ is even.}
		\end{cases} 
\end{equation*}   
Here the classes in the second factor arise for each space indexed by $\mathcal{M}^{\text{re}}_F = \mathcal{M}^{\text{ur}}_F \cup \mathcal{M}^{\text{ram}}_F$, while the classes in the third factor appear from each $\frac{\text{SU}_d}{\text{SO}_d}$ only if $d$ is even. The images of the $\alpha$ and $\varepsilon$ classes under the Borel map \eqref{3.4} generates $H^{\ast}\big(Y(\mathsf{K}_f), \mathbb{C}\big)_{\text{tr}}$ as a subalgebra. Note that the Borel map is not defined over $\mathbb{Q}$, and the images of the generators listed above need not lie in rational cohomology. However, the Hecke eigenvalues attached to the trivial summand are rational. Thus, the change of scalars in cohomology gives rise to a natural isomorphism  
\begin{equation}
\label{3.5}
H^{\ast}\big(Y(\mathsf{K}_{f}), \mathbb{Q}\big)_{\text{tr}} \otimes \mathbb{C} \cong H^{\ast}\big(Y(\mathsf{K}_{f}), \mathbb{C}\big)_{\text{tr}}.	
\end{equation}
In particular, the rational subspace $H^{\ast}\big(Y(\mathsf{K}_{f}), \mathbb{Q}\big)_{\text{tr}}$ is also closed under cup product. One can use the same argument for the field extension $\mathbb{Q}_{\ell}$ of $\mathbb{Q}$ to discover that $ H^{\ast}\big(Y(\mathsf{K}_{f}), \mathbb{Q}_{\ell}\big)_{\text{tr}} = H^{\ast}\big(Y(\mathsf{K}_{f}), \mathbb{Q}\big)_{\text{tr}} \otimes \mathbb{Q}_{\ell}$ and this subspace is a $\mathbb{Q}_{\ell}$-subalgebra of $H^{\ast}\big(Y(\mathsf{K}_{f}), \mathbb{Q}_{\ell}\big)$. We conclude the discussion with a special rationality property for the $\varepsilon$-classes that is relevant for our purpose. 
	
\begin{lemma}
\label{lemma3.9}
Let the notation be as above. Then \[\emph{Bo}_{\emph{ct}, \Gamma_{\mathsf{K}_f}}(\varepsilon_{d}^{\sigma}) \in H^{\ast}\big(Y(\mathsf{K}_f), \mathbb{Q}\big)_{\emph{tr}}\] for each $\sigma \in \mathcal{M}_{F,d}^{\emph{ur}}$.  
\end{lemma}
	
\begin{proof}
The principal $\mathsf{K}$-bundles $\Gamma_{\mathsf{K}_f} \backslash \mathsf{G}_{\text{re}} \to Y(\mathsf{K}_f)$ gives rise to a natural map $Y(\mathsf{K}_f) \to B\mathsf{K}$. We consider the pullback of this map in the singular cohomology \[H^{\ast}\big(B\mathsf{K}, \mathbb{C}\big) \to H^{\ast}\big(Y(\mathsf{K}_f), \mathbb{C}\big).\] Now the $\varepsilon$ classes in the compact twin arise as pullback of the integral Euler classes, denoted $\chi$, in the cohomology of $B\mathsf{K}$; see Appendix~\ref{appendixA.2}. The images of these $\chi$ classes under the map above lie in $H^{\ast}\big(Y(\mathsf{K}_f), \mathbb{Q}\big)$. Now the assertion is a consequence of the compatibility between the maps into classifying spaces and the Borel map described in Proposition~\ref{propositionA.3}.   
\end{proof}

\section{Stable subspace of cohomology}
\label{section4}
The primary goal of this section is to identify the stable classes appearing in the cohomology of $Y(\mathsf{K}_f)$. For this purpose, one must define an embedding of $Y(\mathsf{K}_f)$ into a larger arithmetic manifold. Let the notation be as in Section~\ref{section3.4}. For a positive integer $r$, we define a closed embedding of algebraic groups 
\[\iota_{r} : \mathsf{G} \xhookrightarrow{} \text{SL}_{rd^2}\]
by letting $\mathsf{D}$ act on $\mathsf{D}^{\oplus r}$ by coordinate wise left multiplication.  The map $\iota_{r}$ induces a closed embedding of real Lie groups $\mathsf{G}(F \otimes_{\mathbb{Q}} \mathbb{R}) \xhookrightarrow{} \text{SL}_{rd^2}(F \otimes_{\mathbb{Q}} \mathbb{R})$. We choose an inner product on $F_{\sigma}^{rd^2}$ for each $\sigma \in \mathcal{M}_{F}^{\infty}$ so that the image of $K_{\sigma} \subseteq \mathsf{G}(F_{\sigma})$ lies inside $K_{rd^2, \sigma} \subseteq \text{SL}_{rd^2}(F_{\sigma})$ where 
\[K_{rd^2, \sigma} = \begin{cases}
		\text{SO}_{rd^2}, & \sigma \in \mathcal{M}_{F}^{\text{re}};\\
		\text{SU}_{rd^2}, & \sigma \in \mathcal{M}_{F}^{\text{im}};
	\end{cases}\] 
see Section~\ref{section4.1} for details. The maximal order $\mathcal{O}_{\mathsf{D}}$ endows $\mathsf{D}^{\oplus r}$ with a natural integral structure. With respect to this choice of coordinates, the embedding $\iota_{r}$ maps the open compact subgroup $\mathsf{K}_{f}$ into the standard maximal compact $\prod_{v \in \mathcal{M}_{F}^{\text{fin}}} \text{SL}_{rd^2}(\mathcal{O}_v)$. Set 
\[\bar{B}\big(\text{SL}_{rd^2}(\mathcal{O}_F)\big) := \text{SL}_{rd^2}(\mathcal{O}_F) \backslash \text{SL}_{rd^2}(F \otimes_{\mathbb{Q}} \mathbb{R})/\prod_{\sigma \in \mathcal{M}_{F}^{\infty}} K_{rd^2, \sigma}.\]
Here $\text{SL}_{rd^2}(\mathcal{O}_F)$ may contain torsion elements, and the quotient by the discrete subgroup need not be a covering. However, if the order of each torsion element of $\text{SL}_{rd^2}(\mathcal{O}_F)$ is invertible in $\mathds{k}$ then $H^{\ast}\big(\bar{B}(\text{SL}_{rd^2}(\mathcal{O}_F)), \mathds{k} \big)$ computes the group cohomology $H^{\ast}\big(\text{SL}_{rd^2}(\mathcal{O}_F), \mathds{k}\big)$. The map $\iota_{r}$ gives rise to an embedding of arithmetic quotients \[\iota_{r}: Y(\mathsf{K}_{f}) \xhookrightarrow{} \bar{B}\big(\text{SL}_{rd^2}(\mathcal{O}_F)\big).\]
For $n \geq 1$ set $\mathsf{G}_n = \text{Res}^{F}_{\mathbb{Q}} \text{SL}_{n,F}$ and consider the sequence of algebraic groups $\{\mathsf{G}_{n}\}_{n \geq 1}$ as in our discussion on stable cohomology in Appendix~\ref{appendixA.1}. We choose a large enough positive integer $r$ so that $rd^2 > d(F)$ and the maps
\begin{itemize}
\item $\text{Bo}_{\text{ct}, \Gamma_{rd^2}}^{j}: H^{j}\big(\text{CT}\big(\mathsf{G}_{rd^2}(\mathbb{R})\big), \mathbb{C}\big) \to H^{j}\big(\text{SL}_{rd^2}(\mathcal{O}_F), \mathbb{C}\big)$,
\item $H^{j}\big(\text{CT}\big(\mathsf{G}_{\infty}(\mathbb{R})\big), \mathbb{C}\big) \to H^{j}\big(\text{CT}\big(\mathsf{G}_{rd^2}(\mathbb{R})\big), \mathbb{C}\big)$,
\item $H_{j}\big(\text{SL}_{rd^2}(\mathcal{O}_F), \mathds{k}\big) \to H_{j}\big(\text{SL}_{\infty}(\mathcal{O}_F), \mathds{k}\big)$ 
\end{itemize}
induces isomorphisms for each $0 \leq j \leq d(F)$ and for all commutative rings $\mathds{k}$. One can use Borel's estimate for the stability range \cite[9]{borel} and van der Kallen's theorem to prescribe an explicit choice for $r$. We use this integer to propose a convenient notion of the stable subspace of the cohomology. 
	
\begin{definition}
\label{definition4.1}
Let $\mathds{k}$ be a commutative ring so that the order of each torsion element of $\text{SL}_{rd^2}(\mathcal{O}_F)$ is invertible in $\mathds{k}$. Then the \textit{stable subspace} of $H^{\ast}\big(Y(\mathsf{K}_f), \mathds{k}\big)$, denoted $H^{\ast}\big(Y(\mathsf{K}_f), \mathds{k}\big)_{\text{st}}$, is the image of the pullback map   
\[\iota_{r}^{\ast}: H^{\ast}\big(\text{SL}_{rd^2}(\mathcal{O}_F), \mathds{k}\big) \to H^{\ast}\big(Y(\mathsf{K}_f), \mathds{k}\big).\] 	
\end{definition}
	
Note that the stable subspace is a $\mathds{k}$-subalgebra of the cohomology ring $H^{\ast}\big(Y(\mathsf{K}_f), \mathds{k}\big)$. We next determine the stable subspace for complex and $\ell$-adic coefficients.

\subsection{Complex coefficients}
\label{section4.1}
Let $r$ be the fixed positive integer in Definition~\ref{definition4.1}. We consider the coordinate wise left multiplication action of $\text{GL}_{d}(\mathbb{C})$ on $M_d(\mathbb{C})^{\oplus r}$. This action gives rise to a homomorphism of Lie groups 
	\begin{equation}
		\label{4.1}
		T^{(r)} : \text{GL}_d(\mathbb{C}) \to \text{GL}_{rd^2}(\mathbb{C}); \hspace{.3cm} X \mapsto T^{(r)}_{X},\; T^{(r)}_X(A):= XA
	\end{equation}
so that $\text{Tr}(T_X^{(r)}) = rd\text{Tr}(X)$ and $\text{det}(T_X^{(r)}) = \text{det}(X)^{rd}$. One may turn $M_{d}(\mathbb{C})^{\oplus r}$ into a Hermitian space by declaring the natural basis arising from elementary matrices to be orthonormal. Then $T^{(r)}$ above maps $\text{SU}_{d}$ inside $\text{SU}_{rd^2}$. Observe that $T^{(r)}$ is compatible with the underlying real structures of $\text{GL}_{d}(\mathbb{C})$ and $\text{GL}_{rd^2}(\mathbb{C})$ and $T^{(r)}$ carries $\text{SO}_d$ inside $\text{SO}_{rd^2}$. If $d$ is even then there is a standard $\mathbb{R}$-linear embedding of $M_{\frac{d}{2}}(\mathbb{H})$ into $M_{d}(\mathbb{C})$
that identifies $M_{\frac{d}{2}}(\mathbb{H})$ with $\Big\{\begin{pmatrix}
		A & -\bar{B}\\
		B & \bar{A}
	\end{pmatrix} \mid A, B \in M_{\frac{d}{2}}(\mathbb{C}) \Big\} \subseteq M_d(\mathbb{C})$; see \eqref{A.8}. 
Here $M_{\frac{d}{2}}(\mathbb{H})$ is a real form of $M_{d}(\mathbb{C})$, i.e., $M_d(\mathbb{C}) = M_{\frac{d}{2}}(\mathbb{H}) \otimes_{\mathbb{R}} \mathbb{C}$. Since this real form differs from the standard one, the image of $\text{Sp}_{\frac{d}{2}}$ under \eqref{4.1} need not lie inside the usual copy of $\text{SO}_{rd^2} \subseteq \text{SU}_{rd^2}$. However, there exists a unitary matrix $A_r \in \text{U}_{rd^2}$ so that $T^{(r)}\big(\text{Sp}_{\frac{d}{2}}\big) \subseteq  A_r\text{SO}_{rd^2}A_r^{-1}$. Let 
\[\iota_{r}: \text{CT}\big(\mathsf{G}(\mathbb{R})\big) \to \text{CT}\big(\mathsf{G}_{rd^2}(\mathbb{R})\big)\]
denote the embedding attached to $\iota_{r}$ at the level of compact duals. The different components of this map are as follows:  
\begin{equation*}
\begin{aligned}
\iota_{r, \sigma}& : \text{SU}_{d} \to \text{SU}_{rd^2}, \hspace{.1cm} X \to T^{(r)}_X; \hspace{.5cm}(\substack{\sigma \in \mathcal{M}_{F}^{\text{im}}})\\
\iota_{r, \sigma}& : \text{SU}_{d}/\text{SO}_d \to \text{SU}_{rd^2}/\text{SO}_{rd^2},\hspace{.1cm} [X] \to [T^{(r)}_X];\hspace{.5cm}(\substack{\sigma \in \mathcal{M}_{F}^{\text{ur}}})\\
\iota_{r, \sigma}&: \text{SU}_{d}/\text{Sp}_{\frac{d}{2}} \to \text{SU}_{rd^2}/A_r\text{SO}_{rd^2} A_r^{-1},\hspace{.1cm} [X] \to [T^{(r)}_X]. \hspace{.5cm}(\substack{\sigma \in \mathcal{M}_{F}^{\text{ram}}})
\end{aligned}
\end{equation*}
	
For the third case, one needs to observe that there is a homeomorphism 
\[\text{$\frac{\text{SU}_{rd^2}}{A_r\text{SO}_{rd^2} A_r^{-1}} \cong \frac{\text{SU}_{rd^2}}{\text{SO}_{rd^2}}$ given by $[X] \mapsto [A_r^{-1}XA_r]$.}\]
The $\alpha$-classes in the cohomology of $\text{SU}_{rd^2}$  (Appendix~\ref{appendixA.2}) are invariant under the conjugation by an element of $\text{U}_{rd^2}$. Thus the pullback of the $\alpha$-classes in $\text{SU}_{rd^2}/\text{SO}_{rd^2}$ under the homeomorphism above is compatible with the natural projection $\text{SU}_{rd^2} \to \text{SU}_{rd^2}/A_r\text{SO}_{rd^2}A_r^{-1}$. In practice, we identify the target at ramified places with $\text{SU}_{rd^2}/\text{SO}_{rd^2}$ via the homeomorphism above.

\begin{lemma}
\label{lemma4.2}
We have 
\[\iota_{r, \sigma}^{\ast}(\alpha_{2p-1}^{\sigma}) = \begin{cases}
			rd\alpha_{2p-1}^{\sigma}, & \text{$2 \leq p \leq d$;}\\
			0, & \text{$p \geq d+1$}.
		\end{cases} \hspace{.5cm}(\substack{\sigma \in \mathcal{M}_F^{\emph{im}}})\]
Moreover,
\[\iota_{r,\sigma}^{\ast}(\alpha_{2p-1}^{\sigma}) = \begin{cases}
			rd\alpha_{2p-1}^{\sigma}, & \text{$2 \leq p \leq d$;}\\
			0, & \text{$p \geq d+1$,}
		\end{cases} \hspace{.5cm}(\substack{\sigma \in \mathcal{M}_F^{\emph{ur}}, \emph{ $p$ odd}})\]
and 
\[\iota_{r, \sigma}^{\ast}(\alpha_{2p-1}^{\sigma}) = \begin{cases}
			rd\alpha_{2p-1}^{\sigma}, & \text{$2 \leq p \leq d$;}\\
			0, & \text{$p \geq d+1$.}
		\end{cases} \hspace{.5cm}(\substack{\sigma \in \mathcal{M}_F^{\emph{ram}}, \emph{ $p$ odd}})\]  
\end{lemma} 
	
\begin{proof}
Because of the compatibility between the $\alpha$-classes in $\text{SU}$ and its quotients, the second and third parts of the statement directly follow from the first identity. We use the de Rham model for $\alpha$-classes \eqref{A.7} to verify the statement for $\text{SU}$. The derivative of $T^{(r)}$ equals 
\[\mathfrak{T}^{(r)}: \mathfrak{gl}_d(\mathbb{C}) \to \mathfrak{gl}_{rd^2}(\mathbb{C});\hspace{.3cm} X \mapsto \mathfrak{T}^{(r)}_X,\; \mathfrak{T}^{(r)}_X(A) = XA \]
where we identify $\text{End}\big(M_{d}(\mathbb{C})^{\oplus r}\big)$ with $\mathfrak{gl}_{rd^2}(\mathbb{C})$ in the expected manner. In particular $\text{Tr}\big(\mathfrak{T}^{(r)}_X\big) = rd\text{Tr}(X)$. Let $X_1, \ldots, X_{2p-1} \in \mathfrak{su}_d$ for some $p \geq 2$. Note that $\mathfrak{T}^{(r)}_{X_1} \circ \cdots \circ \mathfrak{T}^{(r)}_{X_{2p-1}} = \mathfrak{T}^{(r)}_{X_1 \circ \cdots \circ X_{2p-1}}$. Then, a simple calculation demonstrates that 
\begin{equation*}
\begin{aligned}
\iota_{r, \sigma}^{\ast}\big(\Phi_{2p-1}^{rd^2}\big)(X_1, \ldots, X_{2p-1}) & = \Phi_{2p-1}^{rd^2}\big(\mathfrak{T}^{(r)}_{X_1}, \cdots, \mathfrak{T}^{(r)}_{X_{2p-1}}\big)\\
& = \begin{cases}
rd \Phi_{2p-1}^{d}(X_1, \cdots, X_{2p-1}), & \text{$2 \leq p \leq d$;}\\
0, & \text{$p \geq d+1$.}
\end{cases}
\end{aligned}
\end{equation*}  
\end{proof}
	
The embedding $\iota_{r}$ gives rise to a commutative diagram 
\begin{equation}
\label{4.2}
\begin{tikzcd}
H^{\ast}\big(\text{CT}\big(\mathsf{G}_{rd^2}(\mathbb{R})\big), \mathbb{C}\big) \arrow[r, "\iota_{r}^{\ast}"]\arrow[d, "\text{Bo}^{\ast}_{\text{ct}, \Gamma_{rd^2}}"] & H^{\ast}\big(\text{CT}(\mathsf{G}_{\text{re}}), \mathbb{C}\big)\arrow[d, "\text{Bo}^{\ast}_{\text{ct}, \Gamma_{\mathsf{K}_{f}}}"]\\
H^{\ast}\big(\text{SL}_{rd^2}(\mathcal{O}_F), \mathbb{C}\big) \arrow[r, "\iota_{r}^{\ast}"] & H^{\ast}\big(Y(\mathsf{K}_{f}), \mathbb{C}\big). 
\end{tikzcd}
\end{equation}
By choice of $r$, the left vertical arrow is an isomorphism in the degrees $0 \leq j \leq d(F)$. Moreover, the Borel map in the right vertical arrow is also injective.  Recall the explicit description of the trivial subspace from Section~\ref{section3.4}. For simplicity, we use the same notation for the $\alpha$-classes and their images under the Borel map $\text{Bo}^{\ast}_{\text{ct}, \Gamma_{\mathsf{K}_{f}}}$.

\begin{proposition}
\label{proposition4.3}
The stable subspace $H^{\ast}\big(Y(\mathsf{K}_f), \mathbb{C}\big)_{\emph{st}}$ equals the free exterior subalgebra generated by 
\[\big\{\alpha_{2p-1}^{\sigma} \mid \substack{2 \leq p \leq d, \sigma \in \mathcal{M}^{\emph{im}}_F} \big\} \cup \big\{\alpha_{2p-1}^{\sigma} \mid \substack{\text{$2 \leq p \leq d$, $p$ is odd,} \\ \text{and $\sigma \in \mathcal{M}^{\emph{re}}_F$}}\big\}.\]
In particular $H^{\ast}\big(Y(\mathsf{K}_f), \mathbb{C}\big)_{\emph{st}} \subseteq H^{\ast}\big(Y(\mathsf{K}_f), \mathbb{C}\big)_{\emph{tr}}$.  
\end{proposition}
	
\begin{proof}
By the choice of $r$ we know that $\bigoplus_{j = 0}^{d(F)} H^{j}\big(\text{CT}(\mathsf{G}_{rd^2}(\mathbb{R})), \mathbb{C}\big)$ is spanned by the monomials of degree $\leq d(F)$ consisting of the classes 
\[\big\{\alpha_{2p-1}^{\sigma} \mid \substack{2 \leq p \leq \frac{d(F)+1}{2}, \sigma \in \mathcal{M}^{\text{im}}_F} \big\} \cup \big\{\alpha_{2p-1}^{\sigma} \mid \substack{\text{$2 \leq p \leq \frac{d(F)+1}{2}$, $p$ is odd, and $\sigma \in \mathcal{M}^{\text{re}}_F$}}\big\}.\]
Thus, the first part follows Lemma~\ref{lemma4.2} and the diagram \eqref{4.2}. The second part is an easy consequence of the first part. 
\end{proof}
	
If either $d$ is odd or $F$ is totally imaginary, then there is no $\varepsilon$-class in the trivial cohomology, and the inclusion in the proposition above is an equality.  
	
\subsection{$\ell$-adic coefficients}
\label{section4.2}
This subsection aims to determine the stable subspaces in rational and $\ell$-adic cohomology using Proposition~\ref{proposition4.3}. For $j \geq 2$ we choose a free summand $K_{2j-1}(\mathcal{O}_F)_{\text{fr}}$ of $K_{2j-1}(\mathcal{O}_F)$ so that $K_{2j-1}(\mathcal{O}_F) = K_{2j-1}(\mathcal{O}_F)_{\text{fr}} \oplus K_{2j-1}(\mathcal{O}_F)_{\text{tor}}$. Set $\text{rank}_{\mathbb{Z}} K_{2j-1}(\mathcal{O}_F) = R_j$ and let $\{a_{2j-1}^{(m)} \mid 1 \leq m \leq R_j\}$ be a $\mathbb{Z}$-basis of $K_{2j-1}(\mathcal{O}_F)_{\text{fr}}$. Observe that 
	\[R_j = \begin{cases}
		\lvert \mathcal{M}_{F}^{\text{im}} \rvert, & \text{$j$ is even;}\\
		\lvert \mathcal{M}_{F}^{\text{re}} \rvert + \lvert \mathcal{M}_{F}^{\text{im}} \rvert, & \text{$j$ is odd.}
	\end{cases}\]
Recall the modified Hurewicz map $\overline{\text{Hur}}_{2j-1, \mathbb{Q}}$ from Section~\ref{section2.2} and define 
\[ [a_{2j-1}^{(m)}] := \overline{\text{Hur}}_{2j-1,\mathbb{Q}}(a_{2j-1}^{(m)} \otimes 1) \in H_{2j-1}\big(\text{SL}_{\infty}(\mathcal{O}_F), \mathbb{Q}\big). \hspace{.3cm}(\substack{1 \leq m \leq R_j})\] 
We know that $\{[a_{2j-1}^{(m)}] \mid 1 \leq m \leq R_j\}$ is a basis of $P_{2j-1}\big(\text{SL}_{\infty}(\mathcal{O}_F), \mathbb{Q}\big)$. Let $\{A^{(m)}_{2j-1} \mid 1 \leq m \leq R_j\}$ be a collection of cohomology classes that form a dual basis of $I^{2j-1}\big(\text{SL}_{\infty}(\mathcal{O}_F), \mathbb{Q}\big)$. Now using the stable Borel isomorphism and the explicit description of the cohomology algebra of the stable compact twin  $\text{CT}(\mathsf{G}_{\infty}(\mathbb{R}))$ one sees that $H^{\ast}\big(\text{SL}_{\infty}(\mathcal{O}_F), \mathbb{C}\big)$ is a free exterior algebra generated by \[\big\{\alpha_{2j-1}^{\sigma} \mid \substack{j \geq 2, \sigma \in \mathcal{M}^{\text{im}}_F} \big\} \cup \big\{\alpha_{2j-1}^{\sigma} \mid \substack{\text{$j \geq 2$, $j$ is odd,} \\ \text{and $\sigma \in \mathcal{M}^{\text{re}}_F$}}\big\}\] 
where we are using the same notation for a class and its image under the Borel map. For brevity write $\Sigma_{j} = \mathcal{M}_{F}^{\text{im}}$
if $j$ is even and	$\mathcal{M}_{F}^{\text{re}} \cup \mathcal{M}_{F}^{\text{im}}$ if $j$ is odd. Observe that $\{\alpha_{2j-1}^{\sigma} \mid \sigma \in \Sigma_{j}\}$ gives rise to another basis of $I^{2j-1}\big(\text{SL}_{\infty}(\mathcal{O}_F), \mathbb{C}\big)$. Comparing with the $A$-basis described above we discover, given $N \geq 1$ the $\mathbb{C}$-subalgebra generated by $\{A^{(m)}_{2j-1} \mid \substack{1 \leq m \leq R_j, 2 \leq j \leq N}\}$ equals the the subalgebra generated by $\{\alpha^{\sigma}_{2j-1} \mid \substack{2 \leq j \leq N, \sigma \in \Sigma_{j}}\}$. In particular, the rational cohomology $H^{\ast}\big(\text{SL}_{\infty}(\mathcal{O}_F), \mathbb{Q}\big)$ is a free exterior $\mathbb{Q}$-algebra generated by the rational classes $\bigcup_{j = 2}^{\infty}\{A^{(m)}_{2j-1} \mid 1 \leq m \leq R_j\}$.

\begin{lemma}
\label{lemma4.4}
The stable subspace $H^{\ast}\big(Y(\mathsf{K}_f), \mathbb{Q}\big)_{\emph{st}}$ is a free exterior algebra generated by the classes $\big\{\iota_{r}^{\ast}(A_{2j-1}^{(m)}) \mid \substack{1 \leq m \leq R_j,\; 2 \leq j \leq d} \big\}$. In particular $H^{\ast}\big(Y(\mathsf{K}_f), \mathbb{Q}\big)_{\emph{st}}$ is contained in $H^{\ast}\big(Y(\mathsf{K}_f), \mathbb{Q}\big)_{\emph{tr}}$. 
\end{lemma}
	
\begin{proof}
Lemma~\ref{lemma4.2} implies that the bottom arrow of \eqref{4.2} is injective on the subalgebra generated by $\{\alpha^{\sigma}_{2j-1} \mid \substack{2 \leq j \leq d, \sigma \in \Sigma_{j}}\}$. Thus $H^{\ast}\big(Y(\mathsf{K}_f), \mathbb{Q}\big)_{\text{st}}$ contains the free exterior algebra generated by the classes $\big\{\iota_{r}^{\ast}(A_{2j-1}^{(m)}) \mid \substack{1 \leq m \leq R_j, 2 \leq j \leq d} \big\}$. Moreover the $\mathbb{C}$-span of this subalgebra equals the whole of $H^{\ast}\big(Y(\mathsf{K}_f), \mathbb{C}\big)_{\text{st}}$. Therefore, $H^{\ast}\big(Y(\mathsf{K}_f), \mathbb{Q}\big)_{\text{st}}$ must be equal to the subalgebra generated by the pushforwards of $A$-classes listed above. To prove the second part one notes that $H^{\ast}\big(Y(\mathsf{K}_f), \mathbb{Q}\big) \cap H^{\ast}\big(Y(\mathsf{K}_f), \mathbb{C}\big)_{\text{tr}} = H^{\ast}\big(Y(\mathsf{K}_f), \mathbb{Q}\big)_{\text{tr}}$ and applies Proposition~\ref{proposition4.3}.  
\end{proof}
	
We also note down a precise relation between the stable and trivial cohomology necessary for later discussion. Recall that by Lemma~\ref{lemma3.9} each $\varepsilon$-class lies in $H^{\ast}\big(Y(\mathsf{K}_f), \mathbb{Q}\big)_{\text{tr}}$. For a subset $\mathcal{J} \subseteq \mathcal{M}_{F,d}^{\text{ur}}$, write $\varepsilon_{\mathcal{J}} = \prod_{\sigma \in \mathcal{J}} \varepsilon^{\sigma}_d$. 
	
\begin{lemma}
\label{lemma4.5}
We have
\[H^{\ast}\big(Y(\mathsf{K}_f), \mathbb{Q}\big)_{\emph{tr}} = \oplus_{\mathcal{J} \subseteq \mathcal{M}^{\emph{ur}}_{F,d}} \varepsilon_{\mathcal{J}} H^{\ast}\big(Y(\mathsf{K}_f), \mathbb{Q}\big)_{\emph{st}}.\]  
\end{lemma}
	
\begin{proof}
Note that the RHS is contained in the LHS and it generates the full trivial subspace over $\mathbb{C}$. Therefore, it must equal the whole rational trivial subspace. 
\end{proof}
	
Lemma~\ref{lemma4.4} and Lemma~\ref{lemma4.5} together provides a $\mathbb{Q}$-basis $\mathcal{B}$ for $H^{\ast}\big(Y(\mathsf{K}_f), \mathbb{Q}\big)_{\text{tr}}$ consisting monomials generated by \[\big\{\iota_{r}^{\ast}(A_{2j-1}^{(m)}) \mid \substack{1 \leq m \leq R_j, 2 \leq j \leq d} \big\} \cup \{\varepsilon_{d}^{\sigma} \mid \sigma \in \mathcal{M}_{F,d}^{\text{ur}}\}.\] 
Recall that $H^{\ast}\big(Y(\mathsf{K}_f), \mathbb{Q}_{\ell}\big)_{\text{tr}} = H^{\ast}\big(Y(\mathsf{K}_f), \mathbb{Q}\big)_{\text{tr}} \otimes \mathbb{Q}_{\ell}$ for each $\ell$. We choose a finite collection of primes $S_0$ so that for each $\ell \notin S_0$ the following holds: 
\begin{enumerate}[label=(\roman*), align=left]
\item The discrete group $\text{SL}_{rd^2}(\mathcal{O}_F)$ has no $\ell$-torsion. 
\item The cohomology groups $H^{\ast}\big(Y(\mathsf{K}_f), \mathbb{Z}\big)$ and $\oplus_{j=1}^{d(F)} H^{j}\big(\text{SL}_{\infty}(\mathcal{O}_F), \mathbb{Z}\big)$ have no $\ell$-torsion. Moreover, the $\mathbb{Q}$-basis of $\oplus_{j=1}^{d(F)} H^{j}\big(\text{SL}_{\infty}(\mathcal{O}_F), \mathbb{Q}\big)$ consisting of the $A$-monomials descends to a $\mathbb{Z}_{\ell}$-basis of $\oplus_{j=1}^{d(F)} H^{j}\big(\text{SL}_{\infty}(\mathcal{O}_F), \mathbb{Z}_{\ell}\big)$.
\item The rational classes in $\mathcal{B}$ lie inside $H^{\ast}\big(Y(\mathsf{K}_f), \mathbb{Z}_{\ell}\big) \subseteq H^{\ast}\big(Y(\mathsf{K}_f), \mathbb{Q}_{\ell}\big)$ and the $\mathbb{Z}_{\ell}$-span of $\mathcal{B}$ is a direct summand of $H^{\ast}\big(Y(\mathsf{K}_f), \mathbb{Z}_{\ell}\big)$.  
\end{enumerate}
	
Let $\ell$ be a rational prime outside $S_0$. Note that \[H^{\ast}\big(Y(\mathsf{K}_f), \mathbb{Z}_{\ell}\big)_{\text{tr}} = H^{\ast}\big(Y(\mathsf{K}_f), \mathbb{Z}_{\ell}\big) \cap H^{\ast}\big(Y(\mathsf{K}_f), \mathbb{Q}_{\ell}\big)_{\text{tr}}.\] Therefore, by (iii) above, $H^{\ast}\big(Y(\mathsf{K}_f), \mathbb{Z}_{\ell}\big)_{\text{tr}}$ must equal the free summand generated by $\mathcal{B}$. The trivial submodule is automatically a $\mathbb{Z}_{\ell}$-subalgebra since both the modules in the RHS are closed under the cup product. Now suppose $\mathcal{B}_{\text{st}}$ is the subcollection of $\mathcal{B}$ consisting of monomials generated by $\iota_{r}^{\ast}(A^{(m)}_{2j-1})$-s. The second condition on $S_0$ ensures that $\mathcal{B}_{\text{st}}$ is a $\mathbb{Z}_{\ell}$-basis of $H^{\ast}\big(Y(\mathsf{K}_f), \mathbb{Z}_{\ell}\big)_{\text{st}}$. Moreover, there is a direct sum decomposition    
\begin{equation}
\label{4.3}
H^{\ast}\big(Y(\mathsf{K}_f), \mathbb{Z}_{\ell}\big)_{\text{tr}} = \oplus_{\mathcal{J} \subseteq \mathcal{M}^{\text{ur}}_{F,d}} \varepsilon_{\mathcal{J}} H^{\ast}\big(Y(\mathsf{K}_f), \mathbb{Z}_{\ell}\big)_{\text{st}}
\end{equation}
with $\varepsilon_{\mathcal{J}} = \prod_{\sigma \in \mathcal{J}} \varepsilon^{\sigma}_d$ generalizing Lemma~\ref{lemma4.5} to the $\ell$-adic coefficients.
	
Next, we use the description of the $\ell$-adic cohomology to determine the cohomology with coefficients in $\mathbb{Z}/\ell^n$ for each $n \geq 1$. The condition (ii) above immediately verifies the change of coefficients isomorphism for the stable subspace:  
\[H^{\ast}\big(Y(\mathsf{K}_f), \mathbb{Z}/\ell^n\big)_{\text{st}} \cong H^{\ast}\big(Y(\mathsf{K}_f), \mathbb{Z}_{\ell}\big)_{\text{st}} \otimes_{\mathbb{Z}_{\ell}} \mathbb{Z}/\ell^n. \]
To deduce the statement for the trivial submodule, one needs to make a more careful argument since the image of a nonzero $\mathbb{Z}_{\ell}$-linear operator may vanish after the mod $\ell^n$ reduction. Suppose that $T$ is a double coset Hecke operator supported at a good place $v$. Set $\widetilde{T} := T - \deg(T)$. Let $H^{\ast}\big(Y(\mathsf{K}_{f}), \mathbb{Q}\big) = K_{T}^{\ast} \oplus I_{T}^{\ast}$ denote the degree wise \textit{Fitting decomposition} for the endomorphism $\widetilde{T}$, i.e., $\widetilde{T}$ is nilpotent on $K_T^{\ast}$ and invertible on $I_T^{\ast}$. Now suppose $L^{\ast}_{K,T}$, resp. $L^{\ast}_{I,T}$, is a graded lattice inside $K_T^{\ast}$, resp. $I_{T}^{\ast}$. Set $L_{T}^{\ast} = L_{K,T}^{\ast} \oplus L_{I,T}^{\ast}$. We choose a finite set of nonarchimedean places $S_{T}$ containing $S_0$ and the rational prime below $v$ so that for each $\ell \notin S_{T}$ the following holds: 
\begin{itemize}
\item $L_{T}^{\ast} \otimes \mathbb{Z}_{\ell} = H^{\ast}\big(Y(\mathsf{K}_{f}), \mathbb{Z}_{\ell}\big) \subseteq H^{\ast}\big(Y(\mathsf{K}_{f}), \mathbb{Q}_{\ell}\big)$.
\item The action of $\widetilde{T}$ on $H^{\ast}\big(Y(\mathsf{K}_{f}), \mathbb{Z}_{\ell}\big)$ maps $L^{\ast}_{I,T} \otimes \mathbb{Z}_{\ell}$ onto $L^{\ast}_{I,T} \otimes \mathbb{Z}_{\ell}$.   
\end{itemize} 
In particular, $H^{\ast}\big(Y(\mathsf{K}_{f}), \mathbb{Z}_{\ell}\big) = L_{K,T}^{\ast} \otimes \mathbb{Z}_{\ell} \oplus L_{I,T}^{\ast} \otimes \mathbb{Z}_{\ell}$
is the Fitting decomposition with respect to $\widetilde{T}$. Let $\{T_1, \cdots, T_n\}$ be a finite collection of distinct double coset Hecke operators so that 
\begin{equation*}
H^{\ast}\big(Y(\mathsf{K}_{f}), \mathbb{Q}\big)_{\text{tr}} = \bigcap_{j=1}^{n} K_{T_j}^{\ast} \subsetneq \bigcap_{j=1}^{n-1} K_{T_j}^{\ast} \subsetneq \cdots \subsetneq K_{T_1}^{\ast} \cap K_{T_2}^{\ast} \subsetneq K_{T_1}^{\ast}.
\end{equation*}
Here $\text{rank}_{\mathbb{Z}} \bigcap_{r=1}^{n} L_{K,T_j} = \text{dim}_{\mathbb{Q}} H^{\ast}\big(Y(\mathsf{K}_{f}), \mathbb{Q}\big)_{\text{tr}} = \lvert \mathcal{B} \rvert$. Assume that $\ell \notin S_{T_1} \cup \cdots \cup S_{T_n}$. Then $\mathbb{Z}_{\ell}\mathcal{B} \subseteq H^{\ast}\big(Y(\mathsf{K}_{\emph{f}}), \mathbb{Z}_{\ell}\big)_{\text{tr}} \subseteq \big(\bigcap_{j=1}^{n} L_{K,T_j}^{\ast}\big) \otimes \mathbb{Z}_{\ell}$. Since $\mathbb{Z}_{\ell}\mathcal{B}$ is a direct summand of the cohomology, it follows that all three modules above must be equal. Therefore, $H^{\ast}\big(Y(\mathsf{K}_{\emph{f}}), \mathbb{Z}_{\ell}\big)_{\text{tr}} = \big(\bigcap_{j=1}^{n} L_{K,T_j}^{\ast}\big) \otimes \mathbb{Z}_{\ell}$. Observe that \[H^{\ast}\big(Y(\mathsf{K}_{f}), \mathbb{Z}/\ell^n\big) = L_{K,T}^{\ast} \otimes \mathbb{Z}/\ell^n \oplus L_{I,T}^{\ast} \otimes \mathbb{Z}/\ell^n \hspace{.3cm}(\substack{n \geq 1})\]
is the Fitting decomposition of the cohomology as $\mathbb{Z}/\ell^n[\widetilde{T}]$-module. It follows that $H^{\ast}\big(Y(\mathsf{K}_{\emph{f}}), \mathbb{Z}/\ell^n\big)_{\text{tr}} = \big(\bigcap_{j=1}^{n} L_{K,T_j}^{\ast}\big) \otimes \mathbb{Z}/\ell^n = \mathbb{Z}/\ell^n\mathcal{B}$. In particular, the change of coefficients induces an isomorphism   
\begin{equation}
\label{4.4}
H^{\ast}\big(Y(\mathsf{K}_f), \mathbb{Z}/\ell^n\big)_{\text{tr}} \cong H^{\ast}\big(Y(\mathsf{K}_f), \mathbb{Z}_{\ell}\big)_{\text{tr}} \otimes_{\mathbb{Z}_{\ell}} \mathbb{Z}/\ell^n
\end{equation}
and \eqref{4.3} holds for $\mathbb{Z}/\ell^n$-coefficients also. In sequel, we fix the Hecke operators $\{T_1, \ldots, T_n\}$ and a finite subset $S_1 := S_{T_1} \cup \cdots \cup S_{T_n}$ containing $S_0$ so that the change of coefficients isomorphism \eqref{4.4} holds for each $\ell \notin S_1$.
	
\section{Determination of congruence classes}
\label{section5}
Let the notation be in Section~\ref{section3.4} and $\ell$ be a fixed prime. Suppose that $v$ is a good place of $\mathsf{K}_f$ so that $q_{v} \equiv 1($mod $\ell^n)$ for some $n \geq 1$. As before we fix an isomorphism $K_{v} \cong \text{SL}_{d}(\mathcal{O}_{v})$ using the maximal order $\mathcal{O}_{\mathsf{D}}$. Here the natural quotient map $\text{SL}_{d}(\mathcal{O}_{v}) \twoheadrightarrow \text{SL}_{d}(\mathbb{F}_v)$ induces an isomorphism $H^{\ast}\big(\text{SL}_d(\mathbb{F}_v), \mathbb{Z}/\ell^n\big) \cong H^{\ast}\big(\text{SL}_d(\mathcal{O}_{v}), \mathbb{Z}/\ell^n\big)$. Therefore, it suffices to analyze the following congruence class map:  
\[\text{cong}_{v,n}: H^{\ast}\big(\text{SL}_d(\mathbb{F}_v), \mathbb{Z}/\ell^n\big) \to H^{\ast}\big(Y(\mathsf{K}_f), \mathbb{Z}/\ell^n\big).\]
The first goal of this section is to show that the image of $\text{cong}_{v,n}$ lies inside the stable subspace of $Y(\mathsf{K}_f)$ for large $\ell$. Recall that by choice of the integral structure on $\mathsf{D}^{\oplus r}$ we have $\iota_{r}\big(\text{SL}_d(\mathcal{O}_v)\big) \subseteq \text{SL}_{rd^2}(\mathcal{O}_{v})$ and the induced map $\iota_r: \text{SL}_d(\mathcal{O}_v) \to \text{SL}_{rd^2}(\mathcal{O}_v)$ coincides with the closed embedding of topological groups arising from the left multiplication action of $\text{SL}_d(\mathcal{O}_{v})$ on $M_{d}(\mathcal{O}_{v})^{\oplus r}$. One can reduce $\iota_{r}$ modulo $v$ to obtain injective homomorphism of finite groups
\[\iota_{r}: \text{SL}_{d}(\mathbb{F}_v) \to \text{SL}_{rd^2}(\mathbb{F}_v)\]
that equals the map induced by the left multiplication action of $\text{SL}_d(\mathbb{F}_v)$ on $M_{d}(\mathbb{F}_v)^{\oplus r}$. Let $\ell$ be a prime that is coprime to the orders of torsion elements of $\text{SL}_{rd^2}(\mathcal{O}_F)$. For the group cohomology calculations, one must also assume that $\ell > rd^2$.
The discussion above yields a commutative diagram 
\begin{equation}
\label{5.1}
\begin{tikzcd}
H^{*}\big(\text{SL}_{rd^2}(\mathbb{F}_{v}), \mathbb{Z}/\ell^n\big) \arrow[r, "\iota_{r}^{\ast}"]\arrow[d] & H^{*}\big(\text{SL}_{d}(\mathbb{F}_{v}), \mathbb{Z}/\ell^n\big)\arrow[d, "\text{cong}_{v,n}"]\\
H^{*}\big(\text{SL}_{rd^2}(\mathcal{O}_{F}), \mathbb{Z}/\ell^n\big) \arrow[r, "\iota_{r}^{*}"] & H^{*}\big(Y(\mathsf{K}_{f}), \mathbb{Z}/\ell^n\big)  
\end{tikzcd}
\end{equation}
where the left vertical arrow arises from the reduction map $\text{SL}_{rd^2}(\mathcal{O}_F) \to \text{SL}_{rd^2}(\mathbb{F}_v)$. If one views $\bar{B}\big(\text{SL}_{rd^2}(\mathcal{O}_F)\big)$ as an orbifold, then the diagram above expresses nothing but the compatibility of congruence class maps for covering of orbifolds. To prove that the image of $\text{cong}_{v,n}$ lies inside the stable subspace, it is enough to check that the top arrow of the diagram above is surjective. One proves this result through a detailed study of the cohomology of finite groups of Lie type following the work of Quillen and Venkatesh. We first explain how to reduce the problem above to mod $\ell$ coefficients. Note that this result fails to hold without the hypothesis $q \equiv 1($mod $\ell^n)$.  
	
\begin{lemma}
\label{lemma5.1}	
Let $N$ be a positive integer, $\ell$ be a prime $> N$, and $q = 1($\emph{mod} $\ell^n)$ for some positive integer $n \geq 1$. Suppose that $1 \leq m \leq n$. Then the change of coefficient map attached to $\mathbb{Z}/\ell^n \twoheadrightarrow \mathbb{Z}/\ell^m$ induces an isomorphism 
\[H^{\ast}\big(\emph{SL}_N(\mathbb{F}_q), \mathbb{Z}/\ell^n\big) \otimes \mathbb{Z}/\ell^m \cong H^{\ast}\big(\emph{SL}_N(\mathbb{F}_q), \mathbb{Z}/\ell^m\big).\]   
\end{lemma}
	
\begin{proof}
Let $\mathcal{A}'$ be the standard diagonal torus of $\text{SL}_N$. Then an application of Lemma~3.7 in \cite{venkatesh} shows that the restriction gives rise to an isomorphism 
\[H^{\ast}\big(\text{SL}_N(\mathbb{F}_q), \mathbb{Z}/\ell^m\big) \cong H^{\ast}\big(\mathcal{A}'(\mathbb{F}_q), \mathbb{Z}/\ell^m\big)^{W}\]
for each $1 \leq m \leq n$ where $W$ is the Weyl group of $\text{SL}_N$. Observe that $\mathcal{A}'(\mathbb{F}_q)$ is a product of $\mathbb{F}_q^{\times}$ and $H^{\ast}(\mathbb{F}_q^{\times}, \mathbb{Z}/\ell^n) \otimes \mathbb{Z}/\ell^m \cong H^{\ast}(\mathbb{F}_q^{\times}, \mathbb{Z}/\ell^m)$. Moreover, $\lvert W \rvert$ is coprime to $\ell$, and one can retrieve the $W$-fixed submodule by the standard averaging projection. Therefore, the assertion is an easy consequence of the isomorphism above.   
\end{proof}
	
The second lemma helps us to reduce the problem to the general linear group using the same technique. 
	
\begin{lemma}
\label{lemma5.2}
Let the notation be as Lemma~\ref{lemma5.1}. Then the restriction map 
\[H^{\ast}\big(\emph{GL}_{N}(\mathbb{F}_q), \mathbb{Z}/\ell^m\big) \to H^{\ast}\big(\emph{SL}_{N}(\mathbb{F}_q), \mathbb{Z}/\ell^m\big)\]
is surjective. 
\end{lemma}
	
\begin{proof}
Let $\mathcal{A}$ denote the standard diagonal torus of $\text{GL}_{N}$ and set $\mathcal{A}' = \mathcal{A} \cap \text{SL}_N$. Then restriction yields a commutative diagram
\begin{equation}
\label{5.2}
\begin{tikzcd}
H^{\ast}\big(\text{GL}_{N}(\mathbb{F}_q), \mathbb{Z}/\ell^m\big) \arrow[r] \arrow[d] & H^{\ast}\big(\mathcal{A}(\mathbb{F}_q), \mathbb{Z}/\ell^m\big)\arrow[d]\\
H^{\ast}\big(\text{SL}_{N}(\mathbb{F}_q), \mathbb{Z}/\ell^m\big) \arrow[r] & H^{\ast}\big(\mathcal{A}'(\mathbb{F}_q), \mathbb{Z}/\ell^m\big)
\end{tikzcd}
\end{equation}
where the horizontal arrows are injective and surjective onto the Weyl fixed subspace of the cohomology of the corresponding maximal torus. Since $\lvert W \rvert$ is coprime to $\ell$, it suffices to check that the right arrow of \eqref{5.2} is surjective. If $N = 1$, then $\mathcal{A}'(\mathbb{F}_q)$ is trivial, and there is nothing to prove. Assume that $N \geq 2$. Let $\pi: \mathcal{A}(\mathbb{F}_q) \to \mathbb{F}_{v}^{\times (N-1)}$ denote the projection onto first $N-1$ coordinates. Then $\pi$ maps $\mathcal{A}'(\mathbb{F}_q)$ isomorphically onto $\mathbb{F}_{q}^{\times(N-1)}$. But $\pi^{\ast}: H^{\ast}\big(\mathbb{F}_{q}^{\times(N-1)}, \mathbb{Z}/\ell^m\big) \to H^{\ast}\big(\mathcal{A}'(\mathbb{F}_q), \mathbb{Z}/\ell^m\big)$ factors through the right arrow of \eqref{5.2}. In particular, the right arrow is onto.
\end{proof}
	
The final lemma verifies the surjectivity statement for the general linear group with mod $\ell$-coefficients. 
	
\begin{lemma}
\label{lemma5.3}
Let the notation be as in \eqref{5.1}. The pullback homomorphism  
\[\iota_{r}^{\ast}: H^{\ast}\big(\emph{GL}_{rd^2}(\mathbb{F}_v), \mathbb{Z}/\ell\big) \to H^{\ast}\big(\emph{GL}_{d}(\mathbb{F}_v), \mathbb{Z}/\ell\big)\]
is surjective.  
\end{lemma}
	
\begin{proof}
We use Quillen's theory of characteristic classes for the finite groups \cite{quillen} to perform the computation. Let $E$ be the representation of $\text{GL}_{d}(\mathbb{F}_v)$ defined by the homomorphism $\iota_{r}: \text{GL}_{d}(\mathbb{F}_v) \to \text{GL}_{rd^2}(\mathbb{F}_v)$. Then 
\[\iota_{r}^{\ast}\big(c_{j}(\mathbb{F}_v^{\oplus rd^2})\big) = c_j(E), \hspace{.3cm}\iota_{r}^{\ast}\big(e_{j}(\mathbb{F}_v^{\oplus rd^2})\big) = e_j(E). \hspace{.2cm}(\substack{1 \leq j \leq rd^2})\]
The definition of $\iota_{r}$ shows that $[E] = rd[\mathbb{F}_v^{\oplus d}]$ in the Grothendieck group of $\text{GL}_{d}(\mathbb{F}_v)$. Now, the addition formulas for the characteristic classes in the cohomology of the homotopy fiber (\textit{loc. cit.}, p.563) imply  
\begin{equation*}
c_{j}(E) = rd c_j(\mathbb{F}_v^{d}) + c^{\text{dec}}_{j}(E),\; e_{j}(E) = rd e_j(\mathbb{F}_v^{d}) + e^{\text{dec}}_{j}(E)\hspace{.3cm}(\substack{1 \leq j \leq d})
\end{equation*}
where $c^{\text{dec}}_{j}(E)$ and $e^{\text{dec}}_{j}(E)$ are zero if $j = 1$ and, in general, decomposable elements lying in the subalgebra generated by the elements $\{c_{m}(\mathbb{F}_v^d), e_{m}(\mathbb{F}_v^d) \mid 1 \leq m < j\}$. Since $rd$ is an unit in $\mathbb{Z}/\ell$ a recursive argument shows that $\{c_{j}(\mathbb{F}_v^d), e_{j}(\mathbb{F}_v^d) \mid 1 \leq j \leq  d\}$ lies in the image of $\iota_{r}^{\ast}$. Therefore $\iota_{r}^{\ast}$ is surjective.    
\end{proof}
		
We summarize the full discussion as follows: 
	
\begin{proposition}
\label{proposition5.4}
Let $\ell$ be a prime $> rd^2$ so that $\emph{SL}_{rd^2}(\mathcal{O}_F)$ has no $\ell$-torsion. Suppose that $v$ is a good place of $\mathsf{K}_f$ that satisfies $q_v \equiv 1($\emph{mod} $\ell^n)$. Then the top arrow of \eqref{5.1} is onto. In particular, the image of the congruence class map 
\[\emph{cong}_{v,n}: H^{\ast}\big(\emph{SL}_d(\mathbb{F}_v), \mathbb{Z}/\ell^n\big) \to H^{\ast}\big(Y(\mathsf{K}_f), \mathbb{Z}/\ell^n\big)\]
lies inside the stable subspace $H^{\ast}\big(Y(\mathsf{K}_f), \mathbb{Z}/\ell^n\big)_{\emph{st}}$. 
\end{proposition}
	
\begin{proof}
We use Lemma~\ref{lemma5.2} and Lemma~\ref{lemma5.3} to see that the mod $\ell$ pullback map 
\[\iota_{r}^{\ast}: H^{\ast}\big(\text{SL}_{rd^2}(\mathbb{F}_v), \mathbb{Z}/\ell\big) \to H^{\ast}\big(\text{SL}_{d}(\mathbb{F}_v), \mathbb{Z}/\ell\big)\]
is onto. Now, the first part of the assertion follows from Lemma~\ref{lemma5.1} and Nakayama's lemma. The second part is an easy consequence of the first part and the commutativity of \eqref{5.1}. 
\end{proof}	
	
Let $\ell$ be as in the statement of Proposition~\ref{proposition5.4}. Set $\text{Cong}_{n}$ to be the subalgebra generated by the images of the homomorphisms 
\[\big\{\text{cong}_{v,n} \mid \text{$v$ is good for $\mathsf{K}_f$ and $q_v \equiv 1($mod $\ell^n)$}\big\}. \hspace{.3cm}(\substack{n \geq 1})\] 	
We already know that $\text{Cong}_{n} \subseteq H^{\ast}\big(Y(\mathsf{K}_f), \mathbb{Z}/\ell^n\big)_{\text{st}}$. Proposition~\ref{proposition1.2} in the introduction asserts that this inclusion is equality for large enough $\ell$. Our strategy to prove the equality statement involves constructing actual congruence classes using a basis of the $K$-group. For this purpose, one needs to know the behavior of the elements in $K$-groups under the reduction map
\[\Pi_{j,v}: K_{2j-1}(\mathcal{O}_F) \otimes \mathbb{Z}/\ell \to K_{2j-1}(\mathbb{F}_v) \otimes \mathbb{Z}/\ell. \hspace{.3cm}(\substack{j \geq 2})\] 
We initiate the discussion with a few notations that facilitate the presentation of the main result. A subset $\mathcal{B}$ of an abelian group $A$ is \textit{$\ell$-linearly independent} if $\{x \otimes 1 \lvert x \in \mathcal{B}\}$ is $\mathbb{Z}/\ell$-linearly independent subset of $A \otimes \mathbb{Z}/\ell$. Suppose that $\mathcal{J}$ is a nonempty finite subset of $\mathbb{Z}_{\geq 2}$.  For $j \in \mathcal{J}$ let $\mathcal{B}_j$ be a $\ell$-linearly independent subset of $K_{2j-1}(\mathcal{O}_F)$. Set $\mathcal{B} = \amalg_{j \in \mathcal{J}} \mathcal{B}_j$. Suppose now that $\mathcal{S}$ is a subset of $\mathcal{B}$. Write $\mathcal{S}_j = \mathcal{S} \cap \mathcal{B}_j$. Our goal is to determine the density of places $v$ with $q_v \equiv 1($mod $\ell^n)$ so that $\Pi_{j,v}$ maps the elements of $\mathcal{S}_j$ to nonzero and the elements of $\mathcal{B}_j - \mathcal{S}_j$ to zero for each $j \in \mathcal{J}$. Define 
\begin{equation*} 
\mathcal{X}^{\ell,n}_{\mathcal{S}, \mathcal{B}, \mathcal{J}} = \bigcap_{j \in \mathcal{J}}\Big\{\text{$v \in \mathcal{M}_{F}^{\text{fin}}$ and $q_v \equiv 1($mod $\ell^n)$} \,\Bigl\lvert \,\substack{\Pi_{j,v}(\alpha \otimes 1) \neq 0, \forall \alpha \in \mathcal{S}_j;\\ 
\Pi_{j,v}(\alpha \otimes 1) = 0, \forall \alpha \in \mathcal{B}_j - \mathcal{S}_j}\Big\}.\hspace{.2cm}(\substack{n \geq 1})
\end{equation*}  	
Our main result regarding the reduction map (see \cite{reduction}) is as follows: 
	
\begin{theorem}
\label{theorem5.5} 
Put $J = \max \mathcal{J}$. Let $\ell$ be a prime $>  J +1$ so that $\ell$ does not ramify in $F$. Then 
\[\delta(\mathcal{X}^{\ell,n}_{\mathcal{S}, \mathcal{B},\mathcal{J}}) = \frac{(\ell-1)^{\lvert \mathcal{S}\rvert}}{\ell^{\lvert \mathcal{B} \rvert}}\frac{1}{\ell^{n-1}(\ell-1)}.\]
Here, $\delta(\cdot)$ refers to the natural density of a subset of primes.  \end{theorem}
	
This theorem provides a quantitative generalization of the central observation of Arlettaz and Banaszak (Proposition~1 in Section~2 of \cite{crelle1995}) for a collection of $\ell$-linearly independent elements. Note that a place $v$ satisfying $q_v \equiv 1(\text{mod }\ell^n)$ lies in $\mathcal{X}^{\ell,n}_{\mathcal{S}, \mathcal{B}, \mathcal{J}}$ if and only if the tuple \[(\prescript{}{\ell}{\Pi}_{2j-1,v}(\alpha \otimes 1))_{j \in \mathcal{J}, \alpha \in \mathcal{B}_j}\] takes value in $(\mathbb{Z}/\ell)^{\times \lvert \mathcal{S} \lvert} \times \{0\}^{\lvert \mathcal{B} - \mathcal{S}\rvert} \subseteq (\mathbb{Z}/\ell)^{\lvert \mathcal{B}\rvert}$ for a suitable ordering of coordinates. Now the density of $\{v \mid q_v \equiv 1(\text{mod }\ell^n)\}$ equals $\frac{1}{\ell^{n-1}(\ell-1)}$. The theorem above asserts that the ratio of the densities of  $\mathcal{X}^{\ell,n}_{\mathcal{S}, \mathcal{B},\mathcal{J}}$ and $\{v \mid q_v \equiv 1(\text{mod }\ell^n)\}$ equals the naive probability of an element of $(\mathbb{Z}/{\ell})^{\lvert \mathcal{B} \rvert}$ lying in $(\mathbb{Z}/\ell)^{\times \lvert \mathcal{S} \lvert} \times \{0\}^{\lvert \mathcal{B} - \mathcal{S}\rvert}$. In particular, a collection of $\ell$-linearly independent elements exhibits mutually independent reduction patterns. Our main idea for proving Theorem~\ref{theorem5.5} is to employ Soul\'e's first Chern class map to transform the problem into a question about reduction of Galois cohomology classes with coefficients in the modules $\mathbb{Z}/\ell(j)$. Then, the argument proceeds through an analysis of the reduction map for Galois cohomology classes, utilizing standard techniques from Galois theory, and the final result is a consequence of the Chebotarev density theorem applied to a suitable extension. However, for the current article, it suffices to know that $\mathcal{X}^{\ell,n}_{\mathcal{S}, \mathcal{B},\mathcal{J}}$ is infinite.

Let \[\mathcal{J} = \{2, \ldots, d\}, \hspace{.3cm} \mathcal{B} = \{a_{2j-1}^{(m)} \mid \substack{j \in \mathcal{J}, 1 \leq m \leq R_j}\}.\] 
For $j \in \mathcal{J}$ and $1 \leq m \leq R_j$ set $\mathcal{S}_{j,m} =\{a_{2j-1}^{(m)}\}$, and fix a good place $v_{j,m} \in \mathcal{X}^{\ell,n}_{\mathcal{S}_{j,m}, \mathcal{B},\mathcal{J}}$ for each $(j,m)$. Note that $v_{j,m} \neq v_{j',m'}$ whenever $(j,m) \neq (j', m')$. Now suppose $S_2$ is a finite subset of primes containing the set $S_1$ in Section~\ref{section4.2} so that for each $\ell \notin S_2$ the following holds: 
\begin{itemize}
\item $\ell > rd^2$ and $\ell$ does not ramify in $F$.
\item  Recall the notation $C(\cdot, \cdot)$ from Lemma~\ref{lemma2.2}. Then $\ell \nmid C(X,2j-1)$ for $X = B\text{GL}_{\infty}(\mathcal{O}_F)^{+}$ and each $j \in \mathcal{J}$.  
\end{itemize}   
The second point in the list ensures that the Hurewicz map 
\begin{equation}
\label{5.3}
\overline{\text{Hur}}_{2j-1, \mathbb{Z}/\ell}: K_{2j-1}(\mathcal{O}_F) \otimes \mathbb{Z}/\ell \to P_{2j-1}\big(\text{SL}_{\infty}(\mathcal{O}_F), \mathbb{Z}/\ell\big) \hspace{.3cm}(\substack{j \in \mathcal{J}})
\end{equation}
is an isomorphism. Suppose that $\ell$ is a prime so that $\ell \notin S_2$. For each $v_{j,m}$ as above one considers the canonical reduction $\Pi_{j,m}: \text{SL}_{\infty}(\mathcal{O}_F) \to \text{SL}_{\infty}(\mathbb{F}_{v_{j,m}})$ and its pullback in mod $\ell$ cohomology: 
\[\Pi_{j,m}^{\ast}: H^{\ast}\big(\text{SL}_{\infty}(\mathbb{F}_{v_{j,m}}), \mathbb{Z}/\ell\big) \to H^{\ast}\big(\text{SL}_{\infty}(\mathcal{O}_F), \mathbb{Z}/\ell\big).\]   	
Recall the class $\bar{e}_{j,m} \in H^{2j-1}\big(\text{SL}_{\infty}(\mathbb{F}_{v_{j,m}}), \mathbb{Z}/\ell\big)$ from Section~\ref{section2.2}. Set $\xi_{j,m} := \Pi^{\ast}_{j,m}(\bar{e}_{j,m})$ and let $\xi^{(r)}_{j,m}$ be its image in $H^{2j-1}\big(\text{SL}_{rd^2}(\mathcal{O}_F), \mathbb{Z}/\ell\big)$ under the stabilization isomorphism. Write 
\[\eta_{j,m} := \iota_{r}^{\ast}\big(\xi^{(r)}_{j,m}\big) \in H^{2j-1}\big(Y(\mathsf{K}_f), \mathbb{Z}/\ell\big).\]
If $\bar{e}_{j,m}^{(r)} \in H^{2j-1}\big(\text{SL}_{rd^2}(\mathbb{F}_{v_{j,m}}), \mathbb{Z}/\ell\big)$ is the image of $\bar{e}_{j,m}$ under stabilization then the diagram \eqref{5.1} shows that $\text{cong}_{v_{j,m}, 1}(\bar{e}_{j,m}^{(r)}) = \eta_{j,m}$. In particular $\eta_{j,m}$ is a congruence class and it lies inside $H^{\ast}\big(Y(\mathsf{K}_f), \mathbb{Z}/\ell\big)_{\text{st}}$. We next compare $\eta_{j,m}$ with the $A$-classes in Section~\ref{section4.2}.

\begin{lemma}
\label{lemma5.6}
Let $I^{\ast}\big(Y(\mathsf{K}_f), \mathbb{Z}/\ell\big)_{\emph{st}}$ denote the space of indecomposable elements attached to $H^{\ast}\big(Y(\mathsf{K}_f), \mathbb{Z}/\ell\big)_{\emph{st}}$. Then the images of $\eta_{j,m}$ and $\iota_{r}^{\ast}(A_{2j-1}^{(m)})$ in $I^{\ast}\big(Y(\mathsf{K}_f), \mathbb{Z}/\ell\big)_{\emph{st}}$ agree up to a nonzero scalar multiple.   
\end{lemma}	
	
\begin{proof}
Let $j \in \mathcal{J}$ and $1 \leq m \leq R_j$ be fixed. It suffices to verify that the images of $\xi_{j,m}$ and $A^{(m)}_{2j-1}$ in $I^{\ast}\big(\text{SL}_{\infty}(\mathcal{O}_F), \mathbb{Z}/\ell\big)$ agree up to a nonzero scalar multiple. The isomorphism \eqref{5.3} shows that the collection of integral homology classes $\{[a_{2j-1}^{(s)}] \mid 1 \leq s \leq R_j\}$ is a basis of $P_{2j-1}\big(\text{SL}_{\infty}(\mathcal{O}_F), \mathbb{Z}/\ell\big)$. The corresponding dual basis for the space of indecomposable elements is given by the collection $\{A_{2j-1}^{(s)} \mid 1 \leq s \leq R_j\}$. Therefore, to prove the assertion, one needs to check that $\langle [a_{2j-1}^{s}], \xi_{j,m} \rangle \neq 0$ if and only if $s = m$, where $\langle \cdot, \cdot \rangle$ refers to the pairing between primitive and indecomposable elements. But $\langle [a_{2j-1}^{s}], \xi_{j,m} \rangle = \langle \big(\Pi_{j,m}\big)_{\ast}([a_{2j-1}^{s}]), \bar{e}_{j,m} \rangle$. Now, the assertion follows from Corollary~\ref{corollary2.5} and the commutativity of the following diagram: 
\begin{equation*}
\begin{tikzcd}[column sep=large]
K_{2j-1}(\mathcal{O}_F) \otimes \mathbb{Z}/\ell \arrow[r, "\overline{\text{Hur}}_{2j-1, \mathbb{Z}/\ell}"]\arrow[d, "\Pi_{j,v_{j,m}}"] & H_{2j-1}\big(\text{SL}_{\infty}(\mathcal{O}_F), \mathbb{Z}/\ell\big)\arrow[d, "(\Pi_{j,m})_{\ast}"]\\
K_{2j-1}(\mathbb{F}_{v_{j,m}}) \otimes \mathbb{Z}/\ell \arrow[r, "\overline{\text{Hur}}_{2j-1, \mathbb{Z}/\ell}"] & H_{2j-1}\big(\text{SL}_{\infty}(\mathbb{F}_{v_{j,m}}), \mathbb{Z}/\ell\big).  
\end{tikzcd}
\end{equation*}
\end{proof}
	
\begin{corollary}
\label{corollary5.7}
The $\eta$-classes $\{\eta_{j,m} \mid \substack{j \in \mathcal{J}, 1 \leq m \leq R_j}\}$ generate the stable subspace $H^{\ast}\big(Y(\mathsf{K}_f), \mathbb{Z}/\ell\big)_{\emph{st}}$ as a free exterior algebra.  
\end{corollary}
	
\begin{proof}
One knows that the classes $\{\iota_{r}^{\ast}(A_{2j-1}^{(m)}) \mid j \in \mathcal{J},\; 1 \leq m \leq R_j\}$ generate the stable subspace as a free exterior algebra. Thus, the statement is a consequence of  Lemma~\ref{lemma5.6}.  
\end{proof}
	
We next amplify these mod $\ell$ computations to arrive at Proposition~\ref{proposition1.2} in the introduction.
	
\paragraph{Proof of Proposition~\ref{proposition1.2}.} Let $S = S_2$ be the finite collection of primes constructed below Theorem~\ref{theorem5.5} and assume that $\ell \notin S$. Then, Corollary~\ref{corollary5.7} verifies that the assertion holds for $n = 1$. Let $n$ be an arbitrary positive integer. One employs Lemma~\ref{lemma5.1} to discover that given $j \in \mathcal{J}$ and $1 \leq m \leq R_j$ there exists a congruence class $\tilde{\eta}_{j,m} \in H^{2j-1}\big(Y(\mathsf{K}_f), \mathbb{Z}/\ell^n\big)_{\text{st}}$ so that the image of $\tilde{\eta}_{j,m}$ under the change of coefficients map \[H^{\ast}\big(Y(\mathsf{K}_f), \mathbb{Z}/\ell^n\big)_{\text{st}} \otimes \mathbb{Z}/\ell \cong H^{\ast}\big(Y(\mathsf{K}_f), \mathbb{Z}/\ell\big)_{\text{st}}\] equals $\eta_{j,m}$. Let $\mathcal{N}$ denote the $\mathbb{Z}/\ell^n$-subalgebra generated by the classes $\{\tilde{\eta}_{j,m} \mid \substack{j \in \mathcal{J}, 1 \leq m \leq R_j}\}$. Note that the image of $\mathcal{N}$ under the isomorphism above equals $H^{\ast}\big(Y(\mathsf{K}_f), \mathbb{Z}/\ell\big)_{\text{st}}$. Now an application of Nakayama's lemma shows that $\mathcal{N} = \text{Cong}_n = H^{\ast}\big(Y(\mathsf{K}_f), \mathbb{Z}/\ell^n\big)_{\text{st}}$. \hfill $\square$

\section{Structure theorems for the derived Hecke action}
\label{section6}
We now assemble the pieces of the theory and provide proof of the main results announced in the introduction.    
	
\subsection{Proof of Theorem~\ref{theorem1.3} and Corollary~\ref{corollary1.4}}
\label{section6.1}
Let $\mathsf{D}$ be a centrally simple division algebra over a number field $F$ and set $\mathsf{G} = \text{SL}_1(\mathsf{D})$. Recall the subset of finite places $S_2$ from Section~\ref{section5}. Let $S = S_2$ as in the proof of Proposition~\ref{proposition1.2} and $\ell$ be a prime number with $\ell \notin S$. By construction, $\ell > d$ and hence does not divide the size of the Weyl group of $\mathsf{G}$. Therefore, $\mathbb{T}'_{\ell}$ is graded commutative and preserves the Hecke-trivial summand of cohomology. Suppose that $\mathbb{T}'_{\ell, \text{tr}}$ is the image of $\mathbb{T}_{\ell}'$ in the endomorphism ring of $H^{\ast}\big(Y(\mathsf{K}_{f}), \mathbb{Z}_{\ell}\big)_{\text{tr}}$.  
	
\begin{lemma}
\label{lemma6.1}
Let the notation be as above. Suppose that $h \in \mathbb{T}_{\ell, \emph{tr}}'$. Then $h(\mathds{1}) \in H^{\ast}\big(Y(\mathsf{K}_f), \mathbb{Z}_{\ell}\big)_{\emph{st}}$. Moreover, \[\text{$h(x) = x \cup h(\mathds{1})$ for each $x \in H^{\ast}\big(Y(\mathsf{K}_f), \mathbb{Z}_{\ell}\big)_{\emph{st}}$.}\]  
\end{lemma}  
	
\begin{proof}
If $h$ is an underived Hecke operator, then $h(\mathds{1}) = \text{deg}(h)\mathds{1}$. Now suppose $h$ is a derived Hecke operator supported at a good place $v$ satisfying $q_v \equiv 1(\text{mod }\ell^n)$ for some $n \geq 1$. Then, Proposition~\ref{proposition1.2} and Proposition~\ref{proposition3.8}(i) together shows, $t_n(\mathds{1}) \in \text{Cong}_n = H^{\ast}\big(Y(\mathsf{K}_f), \mathbb{Z}/\ell^n\big)_{\text{st}}$ for all $t_n \in \mathcal{H}(G_v,K_v)_{\mathbb{Z}/\ell^n}$. It follows that $h(\mathds{1}) \in H^{\ast}\big(Y(\mathsf{K}_f), \mathbb{Z}_{\ell}\big)_{\text{st}}$. To prove the second part, we fix an embedding $\mathbb{Q}_{\ell} \hookrightarrow \mathbb{C}$ and obtain the following injective change of coefficient maps:  
\begin{equation}
\label{6.1}
H^{\ast}\big(Y(\mathsf{K}_f), \mathbb{Z}_{\ell}\big)_{\text{tr}} \hookrightarrow H^{\ast}\big(Y(\mathsf{K}_f), \mathbb{Q}_{\ell}\big)_{\text{tr}} \hookrightarrow H^{\ast}\big(Y(\mathsf{K}_f), \mathbb{C}\big)_{\text{tr}}.
\end{equation}
Now the discussion in Section~\ref{section3.4} implies that at each admissible place the underived Hecke operators act on $H^{\ast}\big(Y(\mathsf{K}_f), \mathbb{Z}_{\ell}\big)_{\text{tr}}$ as multiplication by $h(\mathds{1}) = \text{deg}(h)\mathds{1}$. Let $v$ be a good place with $q_v \equiv 1(\text{mod }\ell^n)$ and $t_n$ be a derived Hecke operator at $v$.  Proposition~\ref{proposition1.2} and Lemma~\ref{lemma3.5} together show that each class in $H^{\ast}\big(Y(\mathsf{K}_f), \mathbb{Z}/\ell^n\big)_{\text{st}}$ is Hecke super-trivial at $v$. Therefore, by Proposition~\ref{proposition3.8}(ii)  $t_n(x) = x \cup t_n(\mathds{1})$ for each $x \in H^{\ast}\big(Y(\mathsf{K}_f), \mathbb{Z}/\ell^n\big)_{\text{st}}$. Proof of the lemma is now clear.  
\end{proof}
	
\paragraph{Proof of Theorem~\ref{theorem1.3}.} Let $S$ be as above. Recall that by construction $H^{\ast}\big(Y(\mathsf{K}_f), \mathbb{Z}_{\ell}\big)_{\text{st}}$ is a subalgebra of the cohomology ring. Therefore, Lemma~\ref{lemma6.1} demonstrates that the action of $\mathbb{T}'_{\ell}$ preserves the stable submodule $H^{\ast}\big(Y(\mathsf{K}_f), \mathbb{Z}_{\ell}\big)_{\text{st}}$. Suppose that $\mathbb{T}'_{\ell, \text{st}}$ is the image of $\mathbb{T}'_{\ell}$ in the endomorphism ring of $H^{\ast}\big(Y(\mathsf{K}_f), \mathbb{Z}_{\ell}\big)_{\text{st}}$. With each $x \in H^{\ast}\big(Y(\mathsf{K}_{f}), \mathbb{Z}_{\ell}\big)_{\text{st}}$ one can associate an endomorphism \[\text{cup}_x: H^{\ast}\big(Y(\mathsf{K}_{f}), \mathbb{Z}_{\ell}\big)_{\text{st}} \to H^{\ast}\big(Y(\mathsf{K}_{f}), \mathbb{Z}_{\ell}\big)_{\text{st}};\hspace{.3cm} y \mapsto y \cup x.\] Let $\text{Cup}_{\ell, \text{st}}$ denote the subalgebra of endomorphisms generated by the right cup product operators. Lemma~\ref{lemma6.1} implies $\mathbb{T}_{\ell, \text{st}}' \subseteq \text{Cup}_{\ell, \text{st}}$. To conclude the proof of part (i), it suffices to check that the inclusion above is an equality. Now, the image of $\mathbb{T}'_{\ell, n}$ in the endomorphism ring of $H^{\ast}\big(Y(\mathsf{K}_f), \mathbb{Z}/\ell^n\big)$ contains $\text{cup}_x$ for each $x \in \text{Cong}_n = H^{\ast}\big(Y(\mathsf{K}_f), \mathbb{Z}/\ell^n\big)_{\text{st}}$. Hence, $\mathbb{T}_{\ell, \text{st}}' = \text{Cup}_{\ell, \text{st}}$ as desired. To prove the second part one needs to note that by construction the stable cohomology satisfies the change of coefficients isomorphism $H^{\ast}\big(Y(\mathsf{K}_f), \mathbb{Q}_{\ell}\big)_{\text{st}} \cong H^{\ast}\big(Y(\mathsf{K}_f), \mathbb{Q}\big)_{\text{st}} \otimes \mathbb{Q}_{\ell}$.  \hfill $\square$
	
\paragraph{Proof of Corollary~\ref{corollary1.4}.} Let $S = S_2$ as in the discussion above. Then the hypothesis of Corollary~\ref{corollary1.4} and \eqref{4.3} together show that $H^{\ast}\big(Y(\mathsf{K}_f), \mathbb{Z}_{\ell}\big)_{\text{tr}} = H^{\ast}\big(Y(\mathsf{K}_f), \mathbb{Z}_{\ell}\big)_{\text{st}}$. Hence, the assertion is a consequence of Theorem~\ref{theorem1.3}. \hfill $\square$

\subsection{Proof of Theorem~\ref{theorem1.5}}
\label{section6.2}
Let the notation be as in the beginning of the previous section, and $\ell \notin S$. One can use Proposition~\ref{proposition1.2} and the arguments in the proof of Theorem~\ref{theorem1.3} to discover that
\begin{equation}
\label{6.2}
\text{Cup}_{\ell, \text{tr}} := \{\text{cup}_x \mid x \in H^{\ast}\big(Y(\mathsf{K}_f), \mathbb{Z}_{\ell}\big)_{\text{st}}\} \subseteq \mathbb{T}'_{\ell, \text{tr}}
\end{equation}
where we are considering $\text{cup}_x$ as an endomorphism of $H^{\ast}\big(Y(\mathsf{K}_f), \mathbb{Z}_{\ell}\big)_{\text{tr}}$. The proof of Lemma~\ref{lemma6.1} shows that to prove \eqref{6.2} is an equality, it suffices to verify all classes in $H^{\ast}\big(Y(\mathsf{K}_f), \mathbb{Z}/\ell^n\big)_{\text{tr}}$ are super-trivial at every good place $v$ satisfying $q_v \equiv 1(\text{mod }\ell^n)$. Let $n \geq 1$ be fixed. By our choice of $\ell$, there is a direct sum decomposition 
\[H^{\ast}\big(Y(\mathsf{K}_f), \mathbb{Z}/\ell^n\big)_{\text{tr}} = \oplus_{\mathcal{J} \subseteq \mathcal{M}^{\text{ur}}_{F,d}} \varepsilon_{\mathcal{J}}H^{\ast}\big(Y(\mathsf{K}_f), \mathbb{Z}/\ell^n\big)_{\text{st}};\]
see Section~\ref{section4.2}. Now, each class in stable cohomology is already Hecke super-trivial at the required places. Thus, it remains to check that each $\varepsilon$ class appearing above is also super-trivial. We perform this calculation using the classical picture for Hecke operators, as opposed to the adelic picture. Let $\gamma \in \mathsf{G}(F)$ and consider the discrete subgroup $\Gamma_{\mathsf{K}_f}^{(\gamma)} := \gamma \Gamma_{\mathsf{K}_{f}} \gamma^{-1} \subseteq \mathsf{G}_{\text{re}}$. Now we have a canonical map 
\[[\gamma]: \Gamma_{\mathsf{K}_{f}}^{(\gamma)} \backslash \mathsf{G}_{\text{re}}/\mathsf{K} \to \Gamma_{\mathsf{K}_{f}} \backslash \mathsf{G}_{\text{re}}/\mathsf{K}, \hspace{.3cm}[g] \mapsto [\gamma^{-1}g]\]
so that the following commutative square is a pullback diagram of principal $\mathsf{K}$-bundles:
\begin{equation*}
\begin{tikzcd}
\Gamma_{\mathsf{K}_{f}}^{(\gamma)} \backslash \mathsf{G}_{\text{re}}\arrow[d]\arrow[r, "{[g]} \mapsto {[\gamma^{-1}g]}"] \arrow[dr, phantom, "\ulcorner", very near start] &	\Gamma_{\mathsf{K}_{f}} \backslash \mathsf{G}_{\text{re}} \arrow[d]\\
\Gamma_{\mathsf{K}_{f}}^{(\gamma)} \backslash \mathsf{G}_{\text{re}}/\mathsf{K}\arrow[r, "{[\gamma]}"] &	\Gamma_{\mathsf{K}_{f}} \backslash \mathsf{G}_{\text{re}}/\mathsf{K},
\end{tikzcd}
\end{equation*} 
that is to say, the pullback of the bundle on the right side by $[\gamma]$ equals the bundle on the left side. Thus, we have a commutative diagram
\begin{equation*}
\begin{tikzcd}
H^{\ast}\big(B\mathsf{K}, \mathbb{Z}/\ell^n\big) \arrow[r] \arrow[dr]& H^{\ast}\big(\Gamma_{\mathsf{K}_{f}} \backslash \mathsf{G}_{\text{re}}/\mathsf{K}, \mathbb{Z}/\ell^n\big) \arrow[d, "{[\gamma]^{\ast}}"]\\
& H^{\ast}\big(\Gamma_{\mathsf{K}_{f}}^{(\gamma)} \backslash \mathsf{G}_{\text{re}}/\mathsf{K}, \mathbb{Z}/\ell^n\big)
\end{tikzcd}
\end{equation*} 
where the right and down-right arrows arise from the natural maps into the classifying space. But the $\varepsilon$ classes in the locally symmetric space arise as the image of Euler classes in the cohomology of $B\mathsf{K}$; see the proof of Lemma~\ref{lemma3.9}. It follows that the $\varepsilon$ classes in the cohomology of $Y(\mathsf{K}_f)$ must be Hecke super-trivial. As a consequence, the inclusion in \eqref{6.2} is an equality, i.e., $\mathbb{T}'_{\ell, \text{tr}} = \text{Cup}_{\ell, \text{tr}}$. Now the first part of the assertion follows from \eqref{4.3}. To prove the second part one additionally needs the change of coefficients isomorphism \[H^{\ast}\big(Y(\mathsf{K}_f), \mathbb{Q}_{\ell}\big)_{\text{tr}} \cong H^{\ast}\big(Y(\mathsf{K}_f), \mathbb{Q}\big)_{\text{tr}} \otimes \mathbb{Q}_{\ell}.\] 
\hfill $\square$

\section*{Acknowledgments}
I am indebted to Prof. Arvind Nair and Prof. Dipendra Prasad for their help and guidance throughout the work. I also wish to thank Prof. Akshay Venkatesh for his encouragement.

\appendix
\section{Cohomology of compact twins} 
\label{appendixA}
This appendix aims to provide a review of Borel's theorem \cite{borel} on stable real cohomology of arithmetic groups that yields an efficient method of computing stable cohomology using the de Rham model.
	
\subsection{Borel's theorem in stable cohomology}
\label{appendixA.1}
Let $\mathsf{G}$ be a connected, semisimple algebraic group over $\mathbb{Q}$ and $\Gamma$ be an arithmetic subgroup of $\mathsf{G}(\mathbb{R})$. For simplicity, assume that $\mathsf{G}(\mathbb{R})$ is a connected Lie group. Now suppose $\mathsf{K}$ is a maximal compact subgroup of $\mathsf{G}(\mathbb{R})$ and set $\mathsf{X} = \mathsf{G}(\mathbb{R})/\mathsf{K}$. Let $\Omega_{\mathsf{X}}^{\ast}$ be the complex of $\mathbb{C}$-valued differential forms on $\mathsf{X}$. The complex $\Omega_{\mathsf{X}}^{\Gamma}$ computes cohomology of the group $\Gamma$ with coefficients in $\mathbb{C}$, i.e., $H^{*}(\Gamma,\mathbb{C}) = H^{*}(\Omega_{\mathsf{X}}^{\Gamma})$. Set $\mathsf{I}_{\mathsf{G}(\mathbb{R})}^{\ast} := \Omega_{\mathsf{X}}^{\mathsf{G}(\mathbb{R})}$. The graded module $\mathsf{I}_{\mathsf{G}(\mathbb{R})}^{*}$ consists of closed harmonic forms and equals the cohomology of the complex $\Omega_{\mathsf{X}}^{\mathsf{G}(\mathbb{R})}$. Thus the inclusion of complexes $\mathsf{I}_{\mathsf{G}(\mathbb{R})}^{\ast} \xhookrightarrow{} \Omega_{\mathsf{X}}^{\Gamma}$ induces a $\mathbb{C}$-algebra homomorphism 
\[\text{Bo}_{\Gamma}^{\ast}: \mathsf{I}_{\mathsf{G}(\mathbb{R})}^{\ast} \to H^{\ast}(\Gamma, \mathbb{C})\]
where algebra structure refers to cup products on domain and codomain. Before proceeding further, we first analyze $\mathsf{I}_{\mathsf{G}(\mathbb{R})}^{\ast}$ in terms of a more familiar object. Let $\mathfrak{g}$, resp. $\mathfrak{k}$, denote the real Lie algebra of the Lie group $\mathsf{G}(\mathbb{R})$, resp. $\mathsf{K}$. One considers the value of the invariant differential forms at identity to obtain an isomorphism of $\mathbb{C}$-algebras $\mathsf{\alpha}_{\mathsf{G}(\mathbb{R})}: \mathsf{I}_{\mathsf{G}(\mathbb{R})}^{\ast} \xrightarrow{\cong} H^{\ast}(\mathfrak{g},\mathfrak{k}, \mathbb{C})$ where the right-hand side stands for the relative Lie algebra cohomology. Suppose that $\mathfrak{g} = \mathfrak{k} \oplus \mathfrak{p}$ is the Cartan decomposition of $\mathfrak{g}$ with respect to $\mathfrak{p}$. Then $\mathfrak{g}^{u} = \mathfrak{k} \oplus i\mathfrak{p} \subseteq \mathfrak{g}_{\mathbb{C}}$ is a compact form of $\mathfrak{g}$ inside the complex Lie algebra. Let $\mathsf{G}(\mathbb{R})^u$ denote the maximal compact subgroup of $\mathsf{G}(\mathbb{C})$ that corresponds to the subalgebra $\mathfrak{g}^{u}$. We define the \textit{compact twin} of $\mathsf{G}(\mathbb{R})$ by \[\text{CT}\big(\mathsf{G}(\mathbb{R})\big) := \mathsf{G}(\mathbb{R})^{u}/\mathsf{K}\] and consider the complex of invariant differential forms $\mathsf{I}^{\ast}_{\mathsf{G}(\mathbb{R})^u}$ as above. Since $\mathsf{G}(\mathbb{R})^{u}$ is compact the inclusion $\mathsf{I}^{\ast}_{\mathsf{G}(\mathbb{R})^{u}} \xhookrightarrow{} \Omega^{\ast}_{\text{CT}(\mathsf{G}(\mathbb{R}))}$ induces a $\mathbb{C}$ algebra isomorphism $\mathsf{I}^{\ast}_{\mathsf{G}(\mathbb{R})^u} \cong H^{\ast}\big(\text{CT}(\mathsf{G}(\mathbb{R})), \mathbb{C}\big)$. As before, the space on the left-hand side admits a canonical description in terms of relative Lie algebra cohomology $\alpha_{\mathsf{G}(\mathbb{R})^u}: \mathsf{I}^{\ast}_{\mathsf{G}(\mathbb{R})^u} \xrightarrow{\cong} H^{\ast}(\mathfrak{g}^u, \mathfrak{k}, \mathbb{C})$. Recall that the complex $\text{Hom}_{\mathfrak{k}}\big(\bigwedge^{\ast}\mathfrak{p}, \mathbb{C}\big)$, resp. $\text{Hom}_{\mathfrak{k}}\big(\bigwedge^{\ast}i\mathfrak{p}, \mathbb{C}\big)$, computes the relative cohomology $H^{\ast}(\mathfrak{g}, \mathfrak{k}, \mathbb{C})$, resp. $H^{\ast}(\mathfrak{g}^{u}, \mathfrak{k}, \mathbb{C})$. Thus, the $\mathbb{R}$-linear map $X \to iX$ on the Lie subalgebra part identifies these two complexes. We assemble these isomorphisms to obtain a $\mathbb{C}$-algebra homomorphism   
\[\text{Bo}_{\text{ct},\Gamma}^{\ast}: H^{\ast}\big(\text{CT}(\mathsf{G}(\mathbb{R})), \mathbb{C}\big) \to H^{\ast}(\Gamma, \mathbb{C})\]
Now Borel's theorem asserts that there is a constant $N(\mathsf{G})$ depending only on the structure of $\mathsf{G}$ so that $\text{Bo}^{j}_{\text{ct}, \Gamma}$ is an isomorphism for each $j \leq N(\mathsf{G})$. Moreover $N(\mathsf{G}) \to \infty$ as $\text{rk}_{\mathbb{Q}}\mathsf{G} \to \infty$. 
	
We explicitly work out the example relevant to this article \cite[9.3]{gil}. Let $F$ be a number field. Set 
\begin{equation}
\label{A.1}
\mathsf{G}_n = \text{Res}^{F}_{\mathbb{Q}} \text{SL}_{n,F}, \hspace{.3cm} \Gamma_n = \text{SL}_n(\mathcal{O}_F). \hspace{.3cm}(\substack{n \geq 1}) 
\end{equation}
Then $\mathsf{G}_n(\mathbb{R}) = \prod_{\sigma \in \mathcal{M}_{F}^{\text{re}}}\text{SL}_n(\mathbb{R}) \times \prod_{\sigma \in \mathcal{M}_{F}^{\text{im}}}\text{SL}_n(\mathbb{C})$ and $\Gamma_n$ is an arithmetic subgroup of $\mathsf{G}_n(\mathbb{R})$. The compact twins of the special linear groups are as follows: 
\begin{equation*}
\text{CT}\big(\text{SL}_n(\mathbb{R})\big) = \frac{\text{SU}_n}{\text{SO}_n}, \hspace{.3cm} \text{CT}\big(\text{SL}_n(\mathbb{C})\big) = \frac{\text{SU}_n \times \text{SU}_n}{\text{SU}_n} \cong \text{SU}_n.
\end{equation*} 
As a consequence \[\text{CT}\big(\mathsf{G}_n(\mathbb{R})\big) = \prod_{\sigma \in \mathcal{M}_{F}^{\text{re}}}\frac{\text{SU}_n} {\text{SO}_n} \times \prod_{\sigma \in \mathcal{M}_{F}^{\text{im}}}\text{SU}_n.\]
Note that $\text{rk}_{\mathbb{Q}} \mathsf{G}_{n} \to \infty$ as $n \to \infty$. We turn $\{\mathsf{G}_n\}_{n \geq 1}$ into a direct system of algebraic groups via the closed embeddings $\text{SL}_{n,F} \xhookrightarrow{} \text{SL}_{n+1,F}$ considered in Section~\ref{section2.1}. The maps in the direct system yield natural embeddings of the corresponding real Lie groups and their compact twins. Write 
\[\text{CT}\big(\mathsf{G}_{\infty}(\mathbb{R})\big) := \varinjlim_{n} \text{CT}\big(\mathsf{G}_{n}(\mathbb{R})\big) = \prod_{\sigma \in \mathcal{M}^{\text{re}}_F} \frac{\text{SU}_{\infty}}{\text{SO}_{\infty}} \times \prod_{\sigma \in \mathcal{M}^{\text{im}}_F} \text{SU}_{\infty}\] 
where the topology on the final space is the weak topology arising from the limit. A standard argument involving CW structures demonstrates that the sequence of spaces $\{\text{CT}\big(\mathsf{G}_n(\mathbb{R})\big)\}_{n \geq 1}$ exhibits stability in homology and cohomology. In particular $\text{CT}\big(\mathsf{G}_{\infty}(\mathbb{R})\big)$ is a nice topological space in the sense of Section~\ref{section2.1}. We combine this result with the stability for discrete subgroups to obtain a stable Borel map
\begin{equation}
\label{A.2}
\text{Bo}^{\ast}_{\text{ct},\infty}: H^{*}\big(\text{CT}\big(\mathsf{G}_{\infty}(\mathbb{R})\big),\mathbb{C}\big) \to H^{*}\big(\text{SL}_{\infty}(\mathcal{O}_F),\mathbb{C}\big).
\end{equation}

\begin{theorem}
\label{theoremA.1}
The map $\emph{Bo}_{\emph{ct},\infty}^{\ast}$ is an isomorphism of graded $\mathbb{C}$-algebras. 
\end{theorem}  
	
\begin{proof}
See Theorem~11.1 in \cite{borel}.   
\end{proof}
	
\subsubsection{Compatibility with the Hopf algebra structure}
\label{appendixA.1.1}
Recall that the CW complex $B\text{SL}_{\infty}(\mathcal{O}_F)^{+}$ has a commutative $H$-group structure arising from operations on the infinite general linear groups \cite{srinivas}. One can use the same operations to turn the infinite compact group $\text{SU}_{\infty}$ and its symmetric space $\text{SU}_{\infty}/\text{SO}_{\infty}$ into commutative $H$-groups \cite[IV.3]{mimura}. As a consequence $\text{CT}\big(\mathsf{G}_{\infty}(\mathbb{R})\big)$ possesses a product $H$-group structure stemming from its components. Next, we show that the stable Borel map is an isomorphism of Hopf algebras. For this purpose, it suffices to check that the Borel map is compatible with the coproducts on its domain and codomain. 
	
The van der Kallen's theorem for the stability implies that the coproduct on $H^{\ast}\big(\text{SL}_{\infty}(\mathcal{O}_F), \mathbb{C}\big)$ is approximated by the naive intertwining map 
\begin{equation}
\label{A.3}
\begin{aligned}
& \text{SL}_m(\mathcal{O}_F) \times \text{SL}_n(\mathcal{O}_F) \to \text{SL}_{m+n}(\mathcal{O}_F), \hspace{.3cm}(\substack{m,n \geq 1})\\
& \hspace{2cm}(A, B) \mapsto \begin{pmatrix}
				A & 0\\
				0 & B
			\end{pmatrix}.
\end{aligned}
\end{equation}
In more detail, given an integer $j \geq 0$ we fix $n_0 \geq 1$ so that for each $n \geq n_0$ the inclusion $\text{SL}_{n}(\mathcal{O}_F) \xhookrightarrow{} \text{SL}_{\infty}(\mathcal{O}_F)$ induces isomorphism in homology and cohomology in the degrees $\leq j$. Then for each $m , n \geq n_0$, the pullback of \eqref{A.3} in cohomology coincides with the pullback of $H$-law in the degrees $\leq j$ since the intertwining map differs from the $H$-law by a pseudo-conjugation which induces the identity map on cohomology; see Remark in \cite[p.26]{srinivas}. We also have a naive intertwining map for compact twins 
\begin{equation}
\label{A.4}
\text{CT}\big(\mathsf{G}_m(\mathbb{R})\big) \times \text{CT}\big(\mathsf{G}_n(\mathbb{R})\big) \to \text{CT}\big(\mathsf{G}_{m+n}(\mathbb{R})\big)
\end{equation}
stemming from the coordinatewise intertwining maps 
\[\text{SU}_m \times \text{SU}_n \to \text{SU}_{m+n}, \hspace{.3cm}\text{SU}_m/\text{SO}_m \times \text{SU}_n/\text{SO}_n \to \text{SU}_{m+n}/\text{SO}_{m+n}\]
given by $(A, B) \mapsto \begin{pmatrix}
		A & 0\\
		0 & B
	\end{pmatrix}$ and its analog.  
An argument similar to the discrete subgroups verifies that in this case also given an integer $j \geq 0$ there exists $n_0 \geq 1$ so that the pullback of $H$-law agrees with the pullback of \eqref{A.4} in degrees $\leq j$ whenever $m,n \geq n_0$. 
	
\begin{theorem}
\label{theoremA.2}\emph{\cite[p.45]{venkatesh}}
The Borel map $\emph{Bo}_{\emph{ct},\infty}^{\ast}$ is an isomorphism of Hopf algebras.
\end{theorem}
	
\begin{proof}
We consider the naive intertwining map 
\begin{equation*}
\mathsf{G}_{m} \times \mathsf{G}_n \to \mathsf{G}_{m+n} \hspace{.3cm}(\substack{m,n \geq 1}) 
\end{equation*}
for algebraic groups arising from the intertwining map $\text{SL}_{m,F} \times \text{SL}_{n,F} \to \text{SL}_{m+n,F}$. This map gives rise to an intertwining map for the corresponding $\mathbb{Q}$ and $\mathbb{R}$-valued points, which in turn induces \eqref{A.3} for discrete subgroups and \eqref{A.4} for the compact twins. Now, an easy exercise using the construction of the Borel map verifies that the following diagram commutes 
\begin{equation*}
\begin{tikzcd}[column sep=huge]
H^{\ast}\big(\text{CT}(\mathsf{G}_{m+n}(\mathbb{R}))\big) \arrow[d]\arrow[r, "\text{Bo}_{\text{ct}, \Gamma_{m+n}}"] & H^{\ast}\big(\Gamma_{m+n}\big)\arrow[d]\\
H^{\ast}\big(\text{CT}(\mathsf{G}_{m}(\mathbb{R}))\big) \otimes H^{\ast}\big(\text{CT}(\mathsf{G}_{n}(\mathbb{R}))\big) \arrow[r, "\text{Bo}_{\text{ct}, \Gamma_{m}} \otimes \text{Bo}_{\text{ct}, \Gamma_{n}}"] & H^{\ast}\big(\Gamma_{m}\big) \otimes H^{\ast}\big(\Gamma_{n}\big)
\end{tikzcd}
\end{equation*} 
where the vertical arrows stem from the intertwining maps. The assertion is a consequence of the commutativity of this diagram and the approximation procedure described above. 
\end{proof}
	
Theorem~\ref{theoremA.2} describes $H^{\ast}\big(\text{SL}_{\infty}(\mathcal{O}_F), \mathbb{C}\big)$ as a Hopf algebra in terms of the more familiar theory of compact Lie groups and their homogeneous spaces. This phenomenon plays a role in the cohomology calculus of Venkatesh's original treatment. We completely bypass the Hopf algebra structure using our systematic approach through congruence classes.

\subsubsection{Borel map and Chern-Weil theory}
\label{appendixA.1.2}
Let the notation be as in the beginning of Section~\ref{appendixA.1}. In particular, we consider differential forms and de Rham cohomology with coefficients in complex vector spaces unless otherwise stated. Moreover, assume that $\Gamma$ is a torsion-free discrete subgroup of $\mathsf{G}(\mathbb{R})$ so that $\Gamma \cap \mathsf{K}$ is trivial. In this situation, $\Gamma \backslash \mathsf{G}(\mathbb{R}) \to \Gamma \backslash \mathsf{G}(\mathbb{R})/\mathsf{K}$ is a principal $\mathsf{K}$-bundle and hence gives rise to a continuous map $\Gamma \backslash \mathsf{G}(\mathbb{R})/\mathsf{K} \to B\mathsf{K}$ unique up to homotopy, where $B\mathsf{K}$ is the classifying space for the compact group $\mathsf{K}$. Similarly, the principal $\mathsf{K}$-bundle $\mathsf{G}(\mathbb{R})^{u} \to \text{CT}\big(\mathsf{G}(\mathbb{R})\big)$ also induces a map $\text{CT}\big(\mathsf{G}(\mathbb{R})\big) \to B\mathsf{K}$. These maps into classifying spaces yield the following maps at the level of singular cohomology:
\begin{equation}
\label{A.5}
\begin{tikzcd}[column sep= large]
H^{\ast}\big(B\mathsf{K}, \mathbb{C}\big) \arrow[r] \arrow[dr] & H^{\ast}\big(\Gamma, \mathbb{C}\big)\\
& H^{\ast}\big(\text{CT}\big(\mathsf{G}(\mathbb{R})\big), \mathbb{C}\big)\arrow[u, dotted, "\text{Bo}^{\ast}_{\text{ct}, \Gamma}"]
\end{tikzcd}
\end{equation}
where the dotted arrow is the Borel homomorphism. The following result is necessary for the structure theorem regarding the derived Hecke action on the full Hecke trivial subspace. 
	
\begin{proposition}
\label{propositionA.3}
The diagram \eqref{A.5} above is commutative. 
\end{proposition}
	
To prove this assertion, we utilize the Chern-Weil theory \cite[Ch.5]{gil}, which describes the cohomology of the classifying space and the map associated with any principal bundle using the de Rham model. Let $K$ be a compact connected Lie group and $G$ be a connected Lie group containing $K$ as a closed subgroup. We consider the principal $K$-bundle $G \to G/K$. Suppose that $\mathfrak{k}$, resp. $\mathfrak{g}$, is the Lie algebra attached to $K$, resp. $G$. Given a connection form $\nabla \in \Omega^{1}\big(G, \mathfrak{k} \otimes_{\mathbb{R}} \mathbb{C}\big)$ for the principal $K$-bundle $G \to G/K$ there is a Chern-Weil homomorphism 
\[f_{\nabla}: W(K, \mathbb{C}) \to \Omega^{\ast}(G)\]
where $W(K, \mathbb{C})$ is the Weil algebra with complex coefficients. The Chern-Weil morphism maps the subalgebra of invariant elements in Weil algebra, denoted $I_{K}^{\ast}$, inside the subalgebra of the basic elements of $\Omega^{\ast}(G)$, which in terms can be identified with $\Omega^{\ast}(G/K)$. Since the left $G$-action and the right $K$-action on $G$ commute with each other, it is possible to choose $\nabla$ to be invariant under the left $G$-action. From now onward, we fix the choice of a $G$-invariant connection form and compute the Chern-Weil morphism using it. Then $f_{\nabla}$ maps $I_{K}^{\ast}$ inside $\Omega^{\ast}(G/K)^G$. Evaluation at the coset of identity in $G/K$ yields a natural isomorphism $\Omega^{\ast}(G/K)^G \cong \Omega^{\ast}(\mathfrak{g}, \mathfrak{k}, \mathbb{C})$ where the complex on the right-hand side is the relative Lie algebra cohomology complex. As a consequence, we obtain a well-defined map $f_{\nabla}: I_{K}^{\ast} \to \Omega^{\ast}(\mathfrak{g}, \mathfrak{k}, \mathbb{C})$. Here the elements of $I_K^{\ast}$ are closed and $f_{\nabla}$ descends to a map of graded algebras $I_K^{\ast} \to H^{\ast}(\mathfrak{g}, \mathfrak{k}, \mathbb{C})$. If $G$ is compact, then a standard averaging argument shows that $H^{\ast}(\mathfrak{g}, \mathfrak{k}, \mathbb{C}) \cong H^{\ast}(G/K, \mathbb{C})$.
	
Now suppose $H$ is a closed Lie subgroup of $G$ containing $K$. One chooses a left invariant connection form on $G$ and restricts it to $H$ to obtain a commutative diagram   
\begin{equation}
\label{A.6}
\begin{tikzcd}[column sep= large]
I_{K}^{\ast} \arrow[r, "f_{\nabla}"] \arrow[dr, "f_{\nabla}"] & \Omega^{\ast}(G/K)^G \arrow[d]\\
& \Omega^{\ast}(H/K)^H
\end{tikzcd} 
\end{equation}
where the vertical arrow is the pullback map. Similarly, if $\Gamma$ is a torsion-free discrete subgroup of $G$, then we push forward the $G$-invariant connection on $G$ to $\Gamma \backslash G \to \Gamma \backslash G /K$ to construct a Chern-Weil morphism using it.
If $G/K$ is contractible, then there is an isomorphism $\Omega^{\ast}(\Gamma \backslash G/K) \cong \Omega^{\ast}(G/K)^{\Gamma}$. In this situation, we have a commutative diagram  
\begin{equation*}	
\begin{tikzcd}[column sep= large]
I_{K}^{\ast} \arrow[r, "f_{\nabla}"] \arrow[dr, "f_{\nabla}"] & \Omega^{\ast}(G/K)^G \arrow[d]\\
& \Omega^{\ast}(\Gamma \backslash G/K) \cong \Omega^{\ast}(G/K)^{\Gamma} 
\end{tikzcd} 
\end{equation*}
where the vertical arrow is the inclusion map. 
	
\textbf{Proof of Proposition~\ref{propositionA.3}.} We fix a left invariant connection $\nabla$ for the principal $\mathsf{K}$-bundle $\mathsf{G}(\mathbb{C}) \to \mathsf{G}(\mathbb{C})/\mathsf{K}$ and pullback to obtain connections for $\mathsf{G}(\mathbb{R}) \to \mathsf{G}(\mathbb{R})/\mathsf{K}$, $\mathsf{G}(\mathbb{R})^{u} \to \mathsf{G}(\mathbb{R})^{u}/\mathsf{K}$, and $\Gamma \backslash \mathsf{G}(\mathbb{R}) \to \Gamma \backslash \mathsf{G}(\mathbb{R})/\mathsf{K}$. Recall that the Chern-Weil map for the universal bundle $E\mathsf{K} \to B\mathsf{K}$ induces an isomorphism $I^{\ast}_{\mathsf{K}} \cong H^{\ast}\big(B\mathsf{K}, \mathbb{C}\big)$. Thus, the Chern-Weil maps attached to $\Gamma \backslash \mathsf{G}(\mathbb{R})/\mathsf{K}$ and $\mathsf{G}(\mathbb{R})^{u}/\mathsf{K}$ compute the solid arrows in \eqref{A.5}. Next we want to compare the Chern-Weil maps for $\big(\mathsf{G}(\mathbb{R}), \mathsf{K}\big)$ and $\big(\mathsf{G}(\mathbb{R})^u, \mathsf{K}\big)$ by relating both of them with the Chern-Weil morphism for $\big(\mathsf{G}(\mathbb{C}), \mathsf{K}\big)$; see \eqref{A.6}. Note that 
\begin{gather*}
\Omega^{\ast}\big(\mathfrak{g}_{\mathbb{C}}, \mathfrak{k}, \mathbb{C}\big) \cong \text{Hom}_{\mathfrak{k}}\big(\bigwedge^{\ast}(\mathfrak{g}_{\mathbb{C}}/\mathfrak{k}), \mathbb{C}\big) \cong \text{Hom}_{\mathfrak{k}}\big(\bigwedge^{\ast}(\mathfrak{p} \oplus i\mathfrak{p} \oplus i\mathfrak{k}), \mathbb{C}\big),\\
\Omega^{\ast}\big(\mathfrak{g}, \mathfrak{k}, \mathbb{C}\big) \cong \text{Hom}_{\mathfrak{k}}\big(\bigwedge^{\ast}\mathfrak{p}, \mathbb{C}\big), \hspace{.3cm}\Omega^{\ast}\big(\mathfrak{g}_u, \mathfrak{k}, \mathbb{C}\big) \cong \text{Hom}_{\mathfrak{k}}\big(\bigwedge^{\ast}i\mathfrak{p}, \mathbb{C}\big). 
\end{gather*}
The pullback $\Omega^{\ast}\big(\mathfrak{g}_{\mathbb{C}}, \mathfrak{k}, \mathbb{C}\big) \to \Omega^{\ast}\big(\mathfrak{g}, \mathfrak{k}, \mathbb{C}\big)$, resp. $\Omega^{\ast}\big(\mathfrak{g}_{\mathbb{C}}, \mathfrak{k}, \mathbb{C}\big) \to \Omega^{\ast}\big(\mathfrak{g}_u, \mathfrak{k}, \mathbb{C}\big)$, is induced by the natural projection of $\mathfrak{p} \oplus i\mathfrak{p} \oplus i\mathfrak{k}$ onto $\mathfrak{p}$, resp. $i\mathfrak{p}$, components. Proof of the assertion now follows from the construction of the Borel map described at the beginning of the section. \hfill $\square$

\subsection{Explicit cohomology rings}
\label{appendixA.2}
We list the cohomology rings of certain compact Lie groups and their homogeneous spaces \cite[III.6]{mimura} that are relevant to this article.  
	
\paragraph{$\text{SU}_n$:} Let $n \geq 1$. An inductive argument using Serre spectral sequence demonstrates that there exists homogeneous integral classes $\{\alpha_{2p-1} \mid 2 \leq p \leq n\}$ with $\text{deg}(\alpha_{2p-1}) = 2p-1$ so that $H^{\ast}\big(\text{SU}_n, \mathbb{Q}\big)$ is isomorphic to the free exterior algebra generated by $\{\alpha_{2p-1} \mid 2 \leq p \leq n\}$. The classes $\alpha_{2p-1}$ are compatible with stabilization, i.e., the pullback of $\alpha_{2p-1}$ along $\text{SU}_{n-1} \xhookrightarrow{} \text{SU}_n$ equals $\alpha_{2p-1} \in H^{\ast}\big(\text{SU}_{n-1}, \mathbb{Q}\big)$ for each $2 \leq p \leq n-1$. In particular, the stable cohomology ring is given by 
$H^{\ast}\big(\text{SU}_{\infty}, \mathbb{Q}\big) \cong \Lambda_{\mathbb{Q}}[\alpha_{2p-1} \mid p \geq 2]$.
We also need an appropriate de Rham model for these classes to perform effective computations. Let $\mathfrak{su}_n \subseteq \mathfrak{gl}_n(\mathbb{C})$ denote the real Lie algebra of $\text{SU}_n$. For each $2 \leq p \leq n$ define $\Phi_{2p-1} \in \text{Hom}_{\mathbb{R}}(\wedge^{2p-1} \mathfrak{su}_n, \mathbb{C})$ by 
\begin{equation}
\label{A.7}
\Phi_{2p-1}(X_1, \ldots, X_{2p-1}) = \sum_{\tau \in S_{2p-1}} \text{sgn}(\tau)\text{Tr}(X_{\tau(1)} \circ \cdots \circ X_{\tau(2p-1)})
\end{equation}
where $S_{2p-1}$ is the symmetric group on $(2p-1)$ letters. The alternating form considered above extends to a closed, bi-invariant form on $\text{SU}_n$ \cite[VI.7]{bookvol3}. Let $[\Phi_{2p-1}]$ denote the cohomology class attached to $\Phi_{2p-1}$. Then there exists a nonzero complex number $C(p)$ depending only on $p$ so that $C(p)[\Phi_{2p-1}]$ equals $\alpha_{2p-1}$ inside $H^{\ast}\big(\text{SU}_n, \mathbb{C}\big)$ \cite[p.48]{gil}.
	
\paragraph{$\text{SU}_n/\text{SO}_n$:} Let $n \geq 1$. There exists rational classes $\{\alpha_{2p-1} \mid 2 \leq p \leq n,\, \text{$p$ odd}\}$ with $\text{deg}(\alpha_{2p-1}) = 2p-1$ so that $H^{\ast}\big(\text{SU}_n/\text{SO}_n, \mathbb{Q}\big)$ contains the free exterior algebra generated by $\{\alpha_{2p-1} \mid 2 \leq p \leq n,\, \text{$p$ odd}\}$ as a subalgebra. If $n$ is odd then $H^{\ast}\big(\text{SU}_n/\text{SO}_n, \mathbb{Q}\big)$ equals the subalgebra described above. Suppose that $n$ is even. Then there exists an Euler class $\chi \in H^{n}\big(B\text{SO}_n, \mathbb{Z}\big)$ that spans a direct summand of the $\mathbb{Z}$-module $H^{\ast}\big(B\text{SO}_n, \mathbb{Z}\big)$. Now the principal $\text{SO}_n$ bundle $\text{SU}_n \to \text{SU}_n/\text{SO}_n$ gives rise to a map $\text{SU}_n/\text{SO}_n \to B\text{SO}_n$. Let $\varepsilon_{n}$ denote the pullback of $\chi$ in the cohomology of $\text{SU}_n/\text{SO}_n$. Then, we have    
\[H^{\ast}\big(\text{SU}_n/\text{SO}_n, \mathbb{Q}\big) \cong \Lambda_{\mathbb{Q}}[\alpha_{2p-1} \mid 2 \leq p \leq n, \text{$p$ odd}] \otimes_{\mathbb{Q}} \Delta[\varepsilon_{n}]\]
where $\Delta[\varepsilon_{n}] = \mathbb{Q} \oplus \mathbb{Q} \varepsilon_{n}$ is a $\mathbb{Q}$-subalgebra of the cohomology ring. Let $\pi^{\text{SO}}_{n}: \text{SU}_n \to \text{SU}_n /\text{SO}_n$ be the natural quotient map. The $\alpha_{2p-1}$ classes for $\text{SU}_n/\text{SO}_n$ are related with their namesakes in the cohomology of $\text{SU}_n$ in the expected manner, i.e., $(\pi^{\text{SO}}_{n})^{\ast}(\alpha_{2p-1}) = \alpha_{2p-1}$ for each $p$ in range. In particular the $\alpha_{2p-1}$ classes for $\text{SU}_n/\text{SO}_n$ are also compatible with stabilization and there is an isomorphism $H^{\ast}\big(\text{SU}_{\infty}/\text{SO}_{\infty}, \mathbb{Q}\big) \cong \Lambda_{\mathbb{Q}}[\alpha_{2p-1} \mid  p \geq 2, \text{ $p$ odd}]$. 
	
\paragraph{$\text{SU}_{2n}/\text{Sp}_n$:} Let $n \geq 1$ and $\mathbb{H}$ denote the quarternion algebra over algebra over $\mathbb{R}$. Write $v \in \mathbb{H}^n$ as $v = a + ib + jc + kd$ and define two real linear maps $z_1, z_2: \mathbb{H}^n \to \mathbb{C}^n$ by $z_1(v) = a+ ib$, $z_2(v) = c - id$ so that $v = z_1(v) + j z_2(v)$. There is a natural $\ast$-compatible embedding of $\mathbb{R}$-algebras
\begin{equation}
\label{A.8}
M_n(\mathbb{H}) \to M_{2n}(\mathbb{C}), \hspace{.3cm} X = z_1(X) +j z_2(X) \mapsto \begin{pmatrix}
z_{1}(X) & - \overline{z_2(X)}\\
z_2(X) & \overline{z_1(X)}
\end{pmatrix}
\end{equation}  
that identifies $M_n(\mathbb{H})$ with the $\mathbb{R}$-subalgebra $\Big\{\begin{pmatrix}
		A & - \bar{B}\\
		B & \bar{A}
	\end{pmatrix} \mid A, B \in M_n(\mathbb{C})\Big\}$ \cite[I.2]{mimura}. 
In this picture the unitary group over the quaternions is $\text{Sp}_n = M_n(\mathbb{H}) \cap \text{SU}_{2n}$. Let $\pi_{n}^{\text{Sp}}: \text{SU}_{2n} \to \text{SU}_{2n}/\text{Sp}_n$ be the natural quotient map. Then the pullback \[\pi_{n}^{\text{Sp}, \ast}: H^{\ast}\big(\text{SU}_{2n}/\text{Sp}_n, \mathbb{Q}\big) \to H^{\ast}\big(\text{SU}_{2n}, \mathbb{Q}\big)\] 
is injective, and its image equals the free subalgebra generated by $\{\alpha_{2p-1} \mid 2 \leq p \leq 2n, \text{$p$ odd}\}$. For convenience, we denote the preimage of $\alpha_{2p-1}$ by the same symbol. Thus $H^{\ast}\big(\text{SU}_{2n}/\text{Sp}_n, \mathbb{Q}\big) \cong \Lambda_{\mathbb{Q}}[\alpha_{2p-1} \mid 2 \leq p \leq 2n, \text{$p$ odd}]$.

\end{document}